\documentclass[
framed_theorems,
framed_proofs,
print,
]{OCPreprint}

\usepackage{enumitem}
\setlist[enumerate,1]{label=\textup{(\roman*)}}

\let\cite\parencite

\newcommand\cnv{\mathbin{*}}

\newcommand\pushforward{\mathbin{\#}}

\title[Second-order conditions for bang-bang control]%
{Second-order conditions for bang-bang control of elliptic equations in arbitrary dimensions}
\author[Wachsmuth]{%
	Gerd Wachsmuth%
	\footnote{%
		Brandenburgische Technische Universität Cottbus--Senftenberg,
		Institute of Mathematics,
		03046 Cottbus,
		Germany,
		\email{gerd.wachsmuth@b-tu.de},
		\url{https://www.b-tu.de/fg-optimale-steuerung},
		ORCID: \href{https://orcid.org/0000-0002-3098-1503}{0000-0002-3098-1503}%
	}
}

\begin{document}
\maketitle

\begin{abstract}
	We consider an optimal control problem governed by a semilinear PDE
	in cases where the optimal control is of bang-bang type.
	By utilizing the theory of Bessel potential space,
	we characterize quadratic growth of the objective
	via a second-order optimality condition.
	In contrast to previous contributions,
	our method of proof works in arbitrary spatial dimensions.
\end{abstract}


\begin{msc}
	\mscLink{49K20},
	\mscLink{49K30},
	\mscLink{49J52}
\end{msc}

\section{Introduction}
\label{sec:intro}
We consider optimal control problems of the form
\begin{equation*}
	\tag{P}
	\label{eq:P}
	\text{Minimize} \qquad F(u) \qquad
	\text{w.r.t.} \qquad u \in \Uad \coloneqq \set{v \in L^\infty(\Omega) \given \abs{v} \le 1}
	,
\end{equation*}
where $\Omega \subset \R^d$ is open and bounded.
We assume that the objective $F$ has certain smoothing properties
which, in particular, exclude that $F$ contains regularization terms like $\norm{u}_{L^2(\Omega)}^2$.
The example which we have in mind is
\begin{equation}
	\label{eq:obj}
	F(u) \coloneqq \frac12\norm{ S(u) - y_d }_{L^2(\Omega)}^2,
\end{equation}
where $S$ is the solution operator of a semilinear and elliptic PDE,
see \cref{subsec:optimal_control}.

It is known that solutions to \eqref{eq:P} could be of bang-bang type,
i.e., they only attain values at the control bounds $\set{-1, 1}$.
Our goal is to investigate
second-order optimality conditions
at a local minimizer $\bar u$
given that the adjoint state $\bar\varphi \coloneqq \nabla F(u)$
is continuously differentiable.
We want to derive both, necessary and sufficient conditions,
and we want to characterize quadratic growth of $F$ w.r.t.\ the control in the neighborhood of $\bar u$.

We mention the relevant literature which addresses second-order optimality condition
for optimal control problems governed by PDEs
with bang-bang controls.
The first contribution is \cite{Casas2012:1}.
In this work, a sufficient condition for quadratic growth
w.r.t.\ the state variable in $L^2(\Omega)$
is obtained.
This does not need any structural assumptions
but the obtained growth is rather weak and no necessary conditions are available.
By using structural conditions on the adjoint state $\bar\varphi$
(basically $\nabla\bar\varphi \ne 0$ on the set $\set{\bar\varphi = 0}$),
\cite{ChristofWachsmuth2017:1,WachsmuthWachsmuth2021}
prove that quadratic growth in $L^1(\Omega)$ is equivalent to the second-order condition
\begin{equation}
	\label{eq:ssc_bang_bang_intro}
	F''(\bar u) h^2
	+
	\int_{\set{\bar\varphi = 0} \cap \Omega} \frac{\abs{\nabla \bar\varphi}}{2} h^2 \d\HH^{d-1}
	>
	0
	\qquad
	\forall h \in L^2(\HH^{d-1}|_{\set{\bar\varphi = 0} \cap \Omega} ) \setminus \set{0}
	.
\end{equation}
The main tool in the proof are the so-called weak-$\star$ second subderivatives.
This analysis, however, is restricted to dimensions $d \le 3$,
since it requires that the linearized control-to-state map
is bounded from $L^1(\Omega)$ to $L^2(\Omega)$.
Finally,
\cite{WachsmuthWachsmuth2025:1}
shows the local convergence of a semismooth Newton method
if an analogue of the second-order condition
\eqref{eq:ssc_bang_bang_intro}
holds
and this result is not restricted to $d \le 3$.
This raises the question if
\eqref{eq:ssc_bang_bang_intro}
is still a sufficient optimality condition in dimensions $d > 3$.

One crucial ingredient to show the equivalence of \eqref{eq:ssc_bang_bang_intro}
with quadratic growth in $L^1(\Omega)$
is the inequality
\begin{equation}
	\label{eq:main_inequality}
	\int_\Omega \bar\varphi (u - \bar u) \d\lambda^d
	\ge
	C \norm{u - \bar u}_{L^1(\Omega)}^2
	\qquad\forall u \in \Uad
\end{equation}
with $C > 0$.
This inequality holds if
$\bar\varphi \in C^1(\bar\Omega)$
satisfies the structural assumption $\nabla\bar\varphi \ne 0$ on $\set{\bar\varphi = 0}$,
see \cite{CasasWachsmuthWachsmuth2016:1} and \cref{sec:inequalities}.
We mention that inequalities similar to \eqref{eq:main_inequality}
play a crucial role for deriving finite-element error estimates,
see, e.g.,
\cite{DeckelnickHinze2012,Jork2024,Fuica2025,FuicaJork2025}.

We show that the space
$L^1(\Omega)$ in \eqref{eq:main_inequality} can be replaced by the dual space
$L^{\alpha,p}(\R^d)\dualspace$
of a Bessel potential space
(see \cref{sec:bessel})
under appropriate assumptions on $\alpha > 0$ and $p \in (1,\infty)$,
see \cref{thm:main_inequality}.
This allows us
to replace $L^1(\Omega)$ in the second-order analysis
by the space $L_0^{\alpha,p}(\Omega)\dualspace$.
Our main result is that the second-order condition \eqref{eq:ssc_bang_bang_intro}
is equivalent to the quadratic growth of $F$
in the space $L^{\alpha,p}(\Omega)\dualspace$ at $\bar u$,
see \cref{thm:second_order_for_bang_bang}.
In particular, we no longer need the restriction $d \le 3$.

The outline of the paper is as follows.
In \cref{sec:bessel},
we briefly introduce Bessel potential spaces
and study the norm of measures in dual Bessel potential spaces.
As a main result, we obtain \cref{thm:pushforward_dual},
which shows that the dual norm of the pushforward measure $T \pushforward \mu$
is bounded from below by the dual norm of $\mu$ itself
if $T$ is Lipschitz continuous.
These results are used in \cref{sec:inequalities}
to verify that we can replace $L^1(\Omega)$ in \eqref{eq:main_inequality}
by a dual Bessel potential space.
Finally,
\cref{sec:application}
contains the application to the optimal control problem \eqref{eq:P}
and, in particular, our main result \cref{thm:second_order_for_bang_bang}.

Throughout the paper,
we denote by $c$ and $C$ generic, positive constants
whose values can change at every occurrence.

\section{Bessel potential spaces}
\label{sec:bessel}

In this section, we recall the definition and basic properties of Bessel potential spaces,
which are a generalization of classical Sobolev spaces.
More information and further references can be found in
\cite[Chapter~V]{Stein1972},
\cite[Section~2.6]{Ziemer1989},
and
\cite[Section~1.2.4]{AdamsHedberg1996}.
Our main goal is to derive a relation between the norms of $\mu$
and its pushforward measure $T \pushforward \mu$ in dual Bessel potential spaces
if $\mu$ is a (nonnegative) measure and $T$ is Lipschitz continuous,
see \cref{thm:pushforward_dual} below.

We denote by $\FF$ the Fourier transform defined via
\begin{equation*}
	(\FF f)(\xi)
	=
	\int_{\R^d} \e^{-\mathrm{i} \xi \cdot x} \d\lambda^d(x)
	\qquad\forall \xi \in \R^d
\end{equation*}
for sufficiently nice functions $f$,
e.g., $f \in L^1(\R^d)$.
Note that this definition implies
$\FF( f \cnv g ) = \FF(f) \FF(g)$ for $f, g \in L^1(\R^d)$.
The Bessel kernel $g_\alpha \colon \R^d \to \R$, $\alpha > 0$, is defined
as the function whose Fourier transform satisfies
\begin{equation*}
	(\FF g_\alpha)(\xi) = (1 + \abs{\xi}^2)^{-\alpha/2}
	.
\end{equation*}
Note that $g_\alpha \in L^1(\R^d)$ for all $\alpha > 0$
and $g_\alpha \cnv g_\beta = g_{\alpha + \beta}$ for all $\alpha,\beta > 0$.
Further,
$\SS(\R^d) \ni \varphi \mapsto g_\alpha \cnv \varphi \in \SS(\R^d)$
is a bijection, where $\SS(\R^d)$ is the (real-valued) Schwartz space,
since
\begin{equation*}
	g_\alpha \cnv \varphi = \FF^{-1}( \FF(g_\alpha) \FF(\varphi))
\end{equation*}
and all operations on the right-hand side are bijections on the (complex-valued)
Schwartz space.
At this point, we mention that all Banach spaces that appear in this paper are real.
Complex numbers only appear in the definition of the Fourier transform
which is used to define the Bessel kernel.

Using the Bessel kernel,
we define the Bessel potential space
$L^{\alpha,p}(\R^d)$
for $\alpha > 0$ and $p \in (1,\infty)$
via
\begin{equation*}
	L^{\alpha, p}(\R^d)
	\coloneqq
	\set{ g_\alpha \cnv f \given f \in L^p(\R^d) }
\end{equation*}
and equip it with the norm
\begin{equation*}
	\norm{g_\alpha \cnv f}_{L^{\alpha,p}(\R^d)}
	\coloneqq
	\norm{f}_{L^p(\R^d)}
	.
\end{equation*}
If $\alpha$ is an integer and $p \in (1, \infty)$,
we recover the usual Sobolev spaces,
see, e.g., \cite[Theorem~1.2.3]{AdamsHedberg1996}.
\begin{theorem}
	\label{thm:bessel_sobolev}
	For all positive integers $k$ and all $p \in (1, \infty)$,
	the Bessel potential space $L^{k,p}(\R^d)$
	equals the Sobolev space $W^{k,p}(\R^d)$
	and the norms in these spaces are equivalent.
\end{theorem}

The density of $\SS(\R^d)$ in $L^p(\R^d)$
implies that $\SS(\R^d) = g_\alpha \cnv \SS(\R^d)$
is dense in $L^{\alpha, p}(\R^d)$.
Via a smooth cut-off argument, we see that
every $\varphi \in \SS(\R^d) \subset W^{k,p}(\R^d)$
can be approximated in $W^{k,p}(\R^d)$ with functions from $C_c^\infty(\R^d)$
(smooth functions with compact support)
and together with $W^{k,p}(\R^d) \embeds L^{\alpha,p}(\R^d)$ for $k \in \N$, $k \ge \alpha$,
this implies that
the space $C_c^\infty(\R^d)$
is dense in $L^{\alpha, p}(\R^d)$.

If $\Omega \subset \R^d$ is an open set,
we define
\begin{equation*}
	L_0^{\alpha,p}(\Omega)
	\coloneqq
	\cl_{L^{\alpha, p}(\R^d)} ( C_c^\infty(\Omega) )
	.
\end{equation*}

In what follows,
we need to study which measures belong to the dual space of $L^{\alpha, p}(\R^d)$.
To this end, we denote by $\MM^+(\R^d)$
the set of all Radon measures,
i.e.,
$\sigma$-additive and $[0,\infty]$-valued measures on the Borel $\sigma$-algebra
which are finite on compact sets,
outer regular on all Borel sets
and inner regular on open sets.
We say that a measure $\mu \in \MM^+(\R^d)$
belongs to $L^{\alpha, p}(\R^d)\dualspace$
if the linear functional
\begin{equation*}
	C_c^\infty(\R^d) \ni \varphi \mapsto \int_{\R^d} \varphi \d\mu \in \R
\end{equation*}
is continuous w.r.t.\ the norm of $L^{\alpha, p}(\R^d)$
or, equivalently,
if this functional can be (uniquely) extended to a functional in $L^{\alpha,p}(\R^d)\dualspace$.
Otherwise, we define $\norm{\mu}_{L^{\alpha,p}(\R^d)\dualspace} \coloneqq \infty$.
The following characterization is well known.
For convenience, we provide the proof.
\begin{lemma}
	\label{lem:dual_measures}
	Let $\alpha > 0$, $p \in (1, \infty)$ and $\mu \in \MM^+(\R^d)$ be given.
	Then, $\mu \in L^{\alpha,p}(\R^d)\dualspace$
	if and only if $g_\alpha \cnv \mu \in L^{p'}(\R^d)$,
	where $p' \in (1, \infty)$ is the conjugate exponent.
	Moreover, we have
	\begin{equation*}
		\norm{\mu}_{L^{\alpha,p}(\R^d)\dualspace}
		=
		\norm{ g_\alpha \cnv \mu }_{L^{p'}(\R^d)}
		.
	\end{equation*}
\end{lemma}
\begin{proof}
	Assume $g_\alpha \cnv \mu \in L^{p'}(\R^d)$.
	Any
	$\varphi \in C_c^\infty(\R^d)$
	can be written as
	$\varphi = g_\alpha \cnv f$ with $f \in L^p(\R^d)$
	and, consequently,
	$\norm{\varphi}_{L^{\alpha,p}(\R^d)} = \norm{f}_{L^p(\R^d)}$.
	Thus, Fubini's theorem
	(which is applicable due to $g_\alpha \cnv \mu \in L^{p'}(\R^d)$ and $f \in L^p(\R^d)$)
	implies
	\begin{equation*}
		\int_{\R^d} \varphi \d\mu
		=
		\int_{\R^d} g_\alpha \cnv f \d\mu
		=
		\int_{\R^d} (g_\alpha \cnv \mu) f \d\lambda^d
		\le
		\norm{g_\alpha \cnv \mu}_{L^{p'}(\R^d)}
		\norm{\varphi}_{L^{\alpha, p}(\R^d)}
		.
	\end{equation*}
	Hence, $\mu \in L^{\alpha,p}(\R^d)\dualspace$.

	For the converse direction,
	we assume $\mu \in L^{\alpha,p}(\R^d)\dualspace$.
	The functional $L^p(\R^d) \ni f \mapsto \dual{\mu}{g_\alpha \cnv f}_{L^{\alpha, p}(\R^d)}$
	is continuous and its norm is $\norm{\mu}_{L^{\alpha,p}(\R^d)\dualspace}$.
	Thus, there exists a function $m \in L^{p'}(\R^d)$ with
	$\norm{m}_{L^{p'}(\R^d)} = \norm{\mu}_{L^{\alpha,p}(\R^d)\dualspace}$
	and
	\begin{equation*}
		\int_{\R^d} m f \d\lambda^d
		=
		\dual{\mu}{g_\alpha \cnv f}_{L^{\alpha, p}(\R^d)}
		\qquad\forall f \in L^p(\R^d).
	\end{equation*}
	It remains to show that $m = g_\alpha \cnv \mu$.
	Let $\varphi \in \SS(\R^d)$ be given.
	Via a smooth cut-off argument,
	there exists a sequence $\seq{\varphi_n} \subset C_c^\infty(\R^d)$
	such that $\varphi_n \to \varphi$ monotonously and in $W^{k,p}(\R^d) \embeds L^{\alpha,p}(\R^d)$,
	with $k \in \N$, $k \ge \alpha$.
	Consequently,
	\begin{equation*}
		\dual{\mu}{\varphi}_{L^{\alpha,p}(\R^d)}
		=
		\lim_{n \to \infty} \dual{\mu}{\varphi_n}_{L^{\alpha,p}(\R^d)}
		=
		\lim_{n \to \infty} \int_{\R^d} \varphi_n \d\mu
		=
		\int_{\R^d} \varphi \d\mu,
	\end{equation*}
	where the third equality uses the monotone convergence theorem.
	In particular, every $\varphi \in \SS(\R^d)$ is $\mu$-integrable.
	This can be applied to $\varphi = g_\alpha \cnv f$
	for an arbitrary $f \in \SS(\R^d)$.
	Consequently,
	\begin{equation*}
		\int_{\R^d} m f \d\lambda^d
		=
		\dual{\mu}{g_\alpha \cnv f}_{L^{\alpha,p}(\R^d)}
		=
		\int_{\R^d} g_\alpha \cnv f \d\mu
		\qquad\forall f \in \SS(\R^d)
		,
	\end{equation*}
	where $m \in L^{p'}(\R^d)$ is as above.
	Now, we additionally suppose that $f \ge 0$
	and we can continue with Fubini's theorem
	(since all involved functions are nonnegative)
	to get
	\begin{equation*}
		\int_{\R^d} m f \d\lambda^d
		=
		\int_{\R^d} g_\alpha \cnv f \d\mu
		=
		\int_{\R^d} (g_\alpha \cnv \mu) f \d\lambda^d
		.
	\end{equation*}
	Since this holds for all $f \in \SS(\R^d)$ with $f \ge 0$,
	we have $g_\alpha \cnv \mu = m \in L^{p'}(\R^d)$.
\end{proof}

The following inequality by Wolff
will be very crucial to estimate the $L^{p'}(\R^d)$ norm of $g_\alpha \cnv \mu$.
For the proof,
we refer to
\cite[Theorem~1]{HedbergWolff1983},
\cite[Theorem~4.5.2 and (4.5.4)]{AdamsHedberg1996}
or
\cite[4.7.5~Theorem]{Ziemer1989}.
\begin{theorem}
	\label{thm:norm_in_dual_space}
	Let $d \in \N$, $\alpha > 0$ and $p \in (1, d/\alpha]$ be given.
	Then, there exist constants $c, C > 0$
	such that
	\begin{equation*}
		c
		\int_{\R^d} (g_\alpha \cnv \mu)^{p'} \d\lambda^d
		\le
		\int_{\R^d}
		\int_0^1
		\parens*{
			\frac{ \mu(B_r(x)) }{r^{d - \alpha p}}
		}^{p' - 1}
		\frac{\d\lambda^1(r)}{r}
		\d\mu(x)
		\le
		C
		\int_{\R^d} (g_\alpha \cnv \mu)^{p'} \d\lambda^d
	\end{equation*}
	holds for all $\mu \in \MM^+(\R^d)$.
\end{theorem}
This directly implies
\begin{equation}
	\label{eq:equivalent_dual_norms}
	c
	\norm{\mu}_{L^{\alpha,p}(\R^d)\dualspace}^{p'}
	\le
	\int_{\R^d}
	\int_0^1
	\parens*{
		\frac{ \mu(B_r(x)) }{r^{d - \alpha p}}
	}^{p' - 1}
	\frac{\d\lambda^1(r)}{r}
	\d\mu(x)
	\le
	C
	\norm{\mu}_{L^{\alpha,p}(\R^d)\dualspace}^{p'}
\end{equation}
for all $\mu \in \MM^+(\R^d)$, see \cref{lem:dual_measures}.

Given a measurable map $T \colon \R^d \to \R^d$,
we consider the pushforward measure $T \pushforward \mu$
which is defined via
\begin{equation*}
	(T \pushforward \mu)(B) \coloneqq \mu(T^{-1}(B))
\end{equation*}
for all Borel sets $B \subset \R^d$.
Note that this directly implies
\begin{equation}
	\label{eq:integral_pushforward}
	\int_{\R^d} f \circ T \d\mu
	=
	\int_{\R^d} f \d(T \pushforward \mu)
\end{equation}
for all Borel measurable $f \colon \R^d \to [0,\infty]$
and for all Borel measurable and $\mu$-integrable $f \colon \R^d \to \R$.

The next result can be considered the main result of this section
and it relates the Bessel dual norms of $\mu$ and $T \pushforward \mu$
in case that $T$ is Lipschitz.

\begin{theorem}
	\label{thm:pushforward_dual}
	Let $d \in \N$, $\alpha > 0$, $p \in (1, d/\alpha]$ and $L > 0$ be given.
	Then, there exists a constant $C > 0$
	such that for all $T \colon \R^d \to \R^d$ which are Lipschitz continuous with constant $L$
	and for all $\mu \in \MM^+(\R^d)$,
	we have
	\begin{equation*}
		\norm{ \mu }_{L^{\alpha, p}(\R^d)\dualspace}
		\le
		C \norm{ T \pushforward \mu }_{L^{\alpha, p}(\R^d)\dualspace}
		.
	\end{equation*}
\end{theorem}
\begin{proof}
	For $x \in \R^d$, we have
	\begin{equation*}
		y \in B_r(x)
		\;\Rightarrow\;
		\abs{ T(x) - T(y) }_{\R^d}
		\le
		L r
		\;\Rightarrow\;
		T(y) \in B_{L r}(T(x)).
	\end{equation*}
	Consequently,
	$B_r(x) \subset T^{-1}(B_{L r}(T(x)))$.
	Thus,
	\begin{align*}
		\MoveEqLeft
		\int_{\R^d}
		\int_0^1
		\parens*{
			\frac{ \mu(B_r(x)) }{r^{d - \alpha p}}
		}^{p' - 1}
		\frac{\d\lambda^1(r)}{r}
		\d\mu(x)
		\\&
		\le
		\int_{\R^d}
		\int_0^1
		\parens*{
			\frac{ \mu(T^{-1}(B_{L r}(T(x)))) }{r^{d - \alpha p}}
		}^{p' - 1}
		\frac{\d\lambda^1(r)}{r}
		\d\mu(x)
		\\&
		=
		\int_{\R^d}
		\int_0^1
		\parens*{
			\frac{ (T \pushforward \mu)(B_{L r}(T(x))) }{r^{d - \alpha p}}
		}^{p' - 1}
		\frac{\d\lambda^1(r)}{r}
		\d\mu(x)
		\\&
		=
		\int_{\R^d}
		\int_0^1
		\parens*{
			\frac{ (T \pushforward \mu)(B_{L r}(x)) }{r^{d - \alpha p}}
		}^{p' - 1}
		\frac{\d\lambda^1(r)}{r}
		\d(T \pushforward \mu)(x)
		,
	\end{align*}
	where we used \eqref{eq:integral_pushforward} in the last step.
	Now, we make the substitution $s = L r$ in the inner integral.
	Thus,
	\begin{equation}
		\label{eq:inequality_in_some_proof}
		\begin{aligned}
			\MoveEqLeft
			\int_{\R^d}
			\int_0^1
			\parens*{
				\frac{ \mu(B_r(x)) }{r^{d - \alpha p}}
			}^{p' - 1}
			\frac{\d\lambda^1(r)}{r}
			\d\mu(x)
			\\&
			\le
			L^{(d - \alpha p)(p' - 1)}
			\int_{\R^d}
			\int_0^L
			\parens*{
				\frac{ (T \pushforward \mu)(B_{s}(x)) }{s^{d - \alpha p}}
			}^{p' - 1}
			\frac{\d\lambda^1(s)}{s}
			\d(T \pushforward \mu)(x)
			.
		\end{aligned}
	\end{equation}
	In case $L \le 1$, the estimate now follows from \cref{thm:norm_in_dual_space},
	since the inner integral can be bounded from above
	by the integral over $[0,1]$.
	In case $L > 1$, we have to show that the inner integral from $1$ to $L$
	can also be estimated by the dual norm of $T \pushforward \mu$.
	For brevity, we define $\nu \coloneqq T \pushforward \mu$.
	Since the Bessel kernel is radially decreasing and positive, there exists $\varepsilon > 0$
	such that $g_\alpha \ge \varepsilon \chi_{B_{2 L}(0)}$.
	For arbitrary $x,y \in \R^d$ with $\abs{x - y} \le L$,
	this implies
	\begin{equation*}
		(g_\alpha \cnv \nu)(y)
		=
		\int_{\R^d} g_\alpha(y - z) \d\nu(z)
		\ge
		\int_{B_{2 L}(y)} \varepsilon \d\nu(z)
		=
		\varepsilon \nu(B_{2 L}(y))
		\ge
		\varepsilon \nu(B_{L}(x))
		.
	\end{equation*}
	With the abbreviations
	$f(y) \coloneqq (g_\alpha \cnv \nu)(y)^{p'-1}$
	and $\kappa \coloneqq (\varepsilon \nu(B_{L}(x)))^{p'-1}$,
	this inequality shows $f \ge \kappa \chi_{B_L(x)}$.
	Consequently,
	\begin{equation*}
		(g_\alpha \cnv f)(x)
		=
		\int_{\R^d} g_\alpha(x - y) f(y) \d\lambda^d(y)
		\ge
		\int_{B_L(x)} g_\alpha(x - y) \kappa \d\lambda^d(y)
		\ge
		\varepsilon \lambda^d(B_L(x)) \kappa,
	\end{equation*}
	i.e.,
	\begin{equation*}
		(g_\alpha \cnv (g_\alpha \cnv \nu)^{p'-1})(x)
		\ge
		\varepsilon^{p'} L^d \lambda^d(B_1(0)) \nu(B_L(x))^{p'-1}
		.
	\end{equation*}
	Now, we integrate this inequality w.r.t.\ $\nu$ and get
	\begin{equation*}
		\int_{\R^d} (g_\alpha \cnv \nu)^{p'} \d\lambda^d
		=
		\int_{\R^d} g_\alpha \cnv (g_\alpha \cnv \nu)^{p'-1} \d\nu
		\ge
		\varepsilon^{p'} L^d \lambda^d(B_1(0)) \int_{\R^d} \nu(B_L(x))^{p'-1} \d\nu(x)
		.
	\end{equation*}
	Consequently,
	\begin{align*}
		\int_{\R^d}
		\int_1^L
		\parens*{
			\frac{ \nu(B_{s}(x)) }{s^{d - \alpha p}}
		}^{p' - 1}
		\frac{\d\lambda^1(s)}{s}
		\d\nu(x)
		&
		\le
		(L - 1)
		\int_{\R^d}
		\nu(B_{L}(x))^{p' - 1}
		\d\nu(x)
		\\&
		\le
		\frac{L - 1}{
			\varepsilon^{p'} L^d \lambda^d(B_1(0))
		}
		\int_{\R^d} (g_\alpha \cnv \nu)^{p'} \d\lambda^d
		.
	\end{align*}
	By combining this estimate with \eqref{eq:inequality_in_some_proof}
	and \cref{thm:norm_in_dual_space},
	we obtain the claim.
\end{proof}
Finally, we mention that the proof in the case $p = 2$ and $L \le 1$
is much simpler and we can avoid \cref{thm:norm_in_dual_space}.
Indeed,
an application of Fubini's theorem implies
\begin{equation*}
	\int_{\R^d} ( g_\alpha \cnv \mu)^2 \d\lambda^d
	=
	\int_{\R^d} g_\alpha \cnv (g_\alpha \cnv \mu) \d\mu
	=
	\int_{\R^d} (g_\alpha \cnv g_\alpha) \cnv \mu \d\mu
	=
	\int_{\R^d} g_{2\alpha} \cnv \mu \d\mu
	.
\end{equation*}
This also applies to $\nu \coloneqq T \pushforward \mu$ and we get
\begin{align*}
	\norm{ T \pushforward \mu }_{L^{\alpha,2}(\R^d)\dualspace}^2
	&=
	\int_{\R^d} ( g_\alpha \cnv \nu)^2 \d\lambda^d
	=
	\int_{\R^d} g_{2\alpha} \cnv \nu \d\nu
	=
	\int_{\R^d} \int_{\R^d} g_{2\alpha}(x - y) \d\nu(y) \d\nu(x)
	\\
	&=
	\int_{\R^d} \int_{\R^d} g_{2\alpha}(T(x) - T(y)) \d\mu(y) \d\mu(x)
	\\
	&\ge
	\int_{\R^d} \int_{\R^d} g_{2\alpha}(x - y) \d\mu(y) \d\mu(x)
	= \ldots =
	\norm{ \mu }_{L^{\alpha,2}(\R^d)\dualspace}^2
	,
\end{align*}
where the inequality follows from the fact that $g_{2\alpha}$ is radially decreasing.
In particular, this shows that we can use the constant $C = 1$ in case $p = 2$ and $L \le 1$.

\section{Growth inequalities}
\label{sec:inequalities}
In this section,
we assume that the following structural condition holds.

\begin{assumption}
	\label{asm:standing}
	The set $\Omega \subset \R^d$, $d \in \N$, is assumed to be open and bounded.
	The function $\bar\varphi \in C^1(\R^d)$
	satisfies $\nabla\bar\varphi \ne 0$ on $\set{\bar\varphi = 0} \cap \bar\Omega$.
\end{assumption}
Note that we assumed that the function $\bar\varphi$ is defined on all of $\R^d$.
In this way, we circumvent some intricacies with possible definitions of $C^1(\bar\Omega)$
if the set $\Omega$ is of low regularity.
In \cref{asm:standing}, we use the notation
$\set{\bar\varphi = 0} = \set{x \in \R^d \given \bar\varphi(x) = 0}$.

In \cref{subsec:subderivatives},
the findings of the current section will be applied to the problem \eqref{eq:P}.
For this, $\bar u \in \Uad$ is a stationary point,
i.e., $\bar\varphi \coloneqq F'(\bar u)$ satisfies
the first-order condition
\begin{equation}
	\label{eq:FONC}
	\int_\Omega \bar\varphi (u - \bar u) \d\lambda^d
	\ge 0
	\qquad\forall u \in \Uad.
\end{equation}
Note that this implies that
$\bar\varphi (u - \bar u) \ge 0$
a.e.\ in $\Omega$ and, consequently,
\begin{equation*}
	\int_\Omega \bar\varphi (u - \bar u) \d\lambda^d
	=
	\int_\Omega \abs{ \bar\varphi (u - \bar u)} \d\lambda^d
	.
\end{equation*}
Together with $\abs{u - \bar u} \le 2$ for all $u \in \Uad$,
we see that
inequality \eqref{eq:main_inequality} is equivalent to
\begin{equation}
	\label{eq:growth_L1}
	\norm{v}_{L^\infty(\Omega)}
	\int_\Omega \abs{\bar\varphi v} \d\lambda^d
	\ge
	c \norm{v}_{L^1(\Omega)}^2
	\qquad
	\forall v \in L^\infty(\Omega)
	.
\end{equation}
In what follows,
we prove that the $L^1(\Omega)$-norm on the right-hand side
can be replaced by suitable norms in dual Bessel potential spaces.

We start by recalling some implications of \cref{asm:standing}.
It is well known that it
implies the existence of a constant $C > 0$
such that
\begin{equation*}
	\lambda^d(\bar\Omega \cap \set{\abs{\bar\varphi} \le \varepsilon}) \le C \varepsilon
	,
\end{equation*}
see \cite[Lemma~3.2]{DeckelnickHinze2012}.
Consequently,
\cite[Proposition~2.7]{CasasWachsmuthWachsmuth2016:1}
implies that \eqref{eq:main_inequality}
and \eqref{eq:growth_L1} hold.
For later reference and for completeness, we prove an one-dimensional version of \eqref{eq:growth_L1}.
\begin{lemma}
	\label{lem:one-dimensional}
	Let $a < b$ be given and assume that the function
	$\psi \in L^\infty(a,b)$ satisfies $\abs{\psi(t)} \ge k \abs{t - \gamma}$ for almost all $t \in (a,b)$
	for some constants $k > 0$ and $\gamma \in (a,b)$.
	Then,
	\begin{equation*}
		\norm{v}_{L^\infty(a,b)}
		\int_a^b \abs{\psi v} \d\lambda^1
		\ge
		\frac{k}{8} \norm{v}_{L^1(a,b)}^2
		\qquad
		\forall v \in L^\infty(a,b)
		.
	\end{equation*}
\end{lemma}
\begin{proof}
	We follow the proof of \cite[Proposition~2.7]{CasasWachsmuthWachsmuth2016:1}.
	For $\varepsilon > 0$, we define $E_\varepsilon \coloneqq \set{ \abs{\psi} \le \varepsilon}$.
	Note that the assumptions imply $\lambda^1(E_\varepsilon) \le 2 \varepsilon / k$.
	For arbitrary $v \in L^\infty(a,b)$ we have
	\begin{equation*}
		\int_a^b \abs{\psi v}\d\lambda^1
		\ge
		\varepsilon \int_{(a,b) \setminus E_\varepsilon} \abs{v} \d\lambda^1
		=
		\varepsilon \parens*{
			\norm{v}_{L^1(a,b)}
			-
			\norm{v}_{L^1(E_\varepsilon)}
		}
		.
	\end{equation*}
	Combining this with the inequality
	$\norm{v}_{L^1(E_\varepsilon)} \le \lambda^1(E_\varepsilon) \norm{v}_{L^\infty(a,b)}$,
	we arrive at
	\begin{equation*}
		\int_a^b \abs{\psi v}\d\lambda^1
		\ge
		\varepsilon \norm{v}_{L^1(a,b)}
		-
		\frac{2}{k} \varepsilon^2 \norm{v}_{L^\infty(a,b)}
		.
	\end{equation*}
	Using $\varepsilon = \alpha \norm{v}_{L^1(a,b)}/\norm{v}_{L^\infty(a,b)}$
	with the optimal choice $\alpha = k/4$
	yields the claim.
\end{proof}

\begin{theorem}
	\label{thm:main_inequality}
	Let \cref{asm:standing} be satisfied.
	Further, let $\alpha > 0$ and $p \in (1, d/\alpha]$ be given
	such that
	$p \ge 2 d / (d - 1 + 2 \alpha)$
	and
	$\alpha > 1/p$.
	Then, there exists a constant $c > 0$
	such that
	\begin{equation}
		\label{eq:main_inequality_bessel}
		\norm{v}_{L^\infty(\Omega)}
		\int_\Omega \abs{\bar\varphi v} \d\lambda^d
		\ge
		c
		\norm{v}_{L^{\alpha,p}(\R^d)\dualspace}^2
	\end{equation}
	holds for all $v \in L^\infty(\Omega)$.
\end{theorem}
Before starting with the proof,
we motivate the underlying idea,
because it is quite hidden within the technicalities.

First of all,
if the function $v$ is zero on a neighborhood of $\set{\bar\varphi = 0}$,
the left-hand side in \eqref{eq:main_inequality_bessel}
can be bounded from below by a constant times $\norm{v}_{L^2(\Omega)}^2$
and this implies \eqref{eq:main_inequality_bessel}.
Consequently,
it will be sufficient to consider a non-negative function $v$ living in the neighborhood of $\set{\bar\varphi = 0}$.
Next, we assume that we have a Lipschitz continuous projection $T$
from this neighborhood to the set $\set{\bar\varphi = 0}$.
We identify $v$ with the measure whose density w.r.t.\ $\lambda^d$ is $v$.
Associated with this measure is the measure $T \pushforward v$,
which lives on the set $\set{\bar\varphi = 0}$.
Then, $T \pushforward v$ is absolutely continuous w.r.t.\ the surface measure $\HH^{d-1}$ on $\set{\bar\varphi = 0}$.
In the proof below, the role of the corresponding density is played by the function $\tilde v$, see \eqref{eq:tilde_v}.
Via \cref{lem:one-dimensional}, we can prove
\begin{equation*}
	\norm{v}_{L^\infty(\Omega)} \int_\Omega \abs{\bar\varphi v }\d\lambda^d
	\ge
	c
	\int_{\set{\bar\varphi = 0}}
	\parens*{ \frac{\d(T \pushforward v)}{\d\HH^{d-1}} }^2
	\d\HH^{d-1}
\end{equation*}
which corresponds to inequality \eqref{eq:ineq_tilde_v}
in the proof below.
This shows that (the density of) $T \pushforward v$
belongs to the trace space $L^2(\HH^{d-1} | _{\set{\bar\varphi = 0}})$.
From a trace theorem, we get that
(under some assumptions on $\alpha$ and $p$)
the trace of functions from
$L^{\alpha,p}(\R^d)$ belongs to $L^2(\HH^{d-1} | _{\set{\bar\varphi = 0}})$.
By dualizing this embedding, we get that
$T \pushforward v$ belongs to the dual space $L^{\alpha,p}(\R^d)\dualspace$
and
\begin{equation*}
	\norm{
		T \pushforward v
	}_{L^{\alpha,p}(\R^d)\dualspace}^2
	\le
	C
	\norm{
		T \pushforward v
	}_{L^2(\HH^{d-1} | _{\set{\bar\varphi = 0}})}^2
	\le
	C
	\norm{v}_{L^\infty(\Omega)} \int_\Omega \abs{\bar\varphi v }\d\lambda^d
	.
\end{equation*}
Using \cref{thm:pushforward_dual},
we can finish with
\begin{equation*}
	\norm{ v }_{L^{\alpha,p}(\R^d)\dualspace}^2
	\le
	C \norm{ T \pushforward v }_{L^{\alpha,p}(\R^d)\dualspace}^2
	\le
	C
	\norm{v}_{L^\infty(\Omega)} \int_\Omega \abs{\bar\varphi v }\d\lambda^d
	.
\end{equation*}
Actually, we do not directly invoke \cref{thm:pushforward_dual}
in the proof, but we employ \eqref{eq:equivalent_dual_norms}.
Due to this, we also circumvent the invocation of the trace theorem.
Now, we make these ideas precise.

\begin{proof}[Proof of \cref{thm:main_inequality}]
	First, we argue that it is sufficient to consider non-negative $v$.
	Suppose that the desired inequality holds for all non-negative $v \in L^\infty(\Omega)$.
	A signed $v$ can be split into its positive and negative part
	$v = v^+ - v^-$.
	Consequently,
	\begin{align*}
		\frac{c}{2} \norm{v}_{L^{\alpha,p}(\R^d)\dualspace}^2
		&\le
		c \norm{v^+}_{L^{\alpha,p}(\R^d)\dualspace}^2
		+
		c \norm{v^-}_{L^{\alpha,p}(\R^d)\dualspace}^2
		\\&
		\le
		\norm{v^+}_{L^\infty(\Omega)}
		\int_\Omega \abs{\bar\varphi v^+} \d\lambda^d
		+
		\norm{v^-}_{L^\infty(\Omega)}
		\int_\Omega \abs{\bar\varphi v^-} \d\lambda^d
		\\&
		\le
		\norm{v}_{L^\infty(\Omega)}
		\parens*{
		\int_\Omega \abs{\bar\varphi v^+} \d\lambda^d
		+
		\int_\Omega \abs{\bar\varphi v^-} \d\lambda^d
		}
		=
		\norm{v}_{L^\infty(\Omega)}
		\int_\Omega \abs{\bar\varphi v} \d\lambda^d
		.
	\end{align*}
	This shows that the inequality holds for all $v \in L^\infty(\Omega)$
	(with a smaller constant).

	Now, let $v \in L^\infty(\Omega)$ be non-negative.
	W.l.o.g., we can assume $\norm{v}_{L^\infty(\Omega)} \le 1$.
	We extend $v$ by $0$ to all of $\R^d$.

	We begin with a local argument.
	Let a point $z \in \set{\bar\varphi = 0} \cap \bar\Omega$ with $\partial_d \bar\varphi(z) > 0$
	be given.
	Due to the implicit function theorem,
	there exist
	an open set $B \subset \R^{d-1}$,
	an open interval $I = (a,b) \subset \R$,
	a function $\gamma \in C^1(B; I)$
	and a constant $\nu \in (0,1]$
	such that
	\begin{align*}
		& z \in B \times I
		, \qquad
		\partial_d \bar\varphi(x,y) \ge \nu \quad\forall (x,y) \in B \times I
		\\
		& \set{\bar\varphi = 0} \cap (B \times I) = \set{(x, \gamma(x)) \given x \in B},
		\\ &
		\set{(x,y) \given x \in B, \abs{y - \gamma(x)} < \nu } \subset B \times I
		.
	\end{align*}
	Note that this implies
	$\bar\varphi(x,a) \le -\nu^2$ and $\bar\varphi(x,b) \ge \nu^2$
	for all $x \in B$.
	We define $\delta \coloneqq \nu^2/2$.
	Assume $v = 0$ on $\R^d \setminus (B \times I)$.
	We define $\tilde v \in L^\infty(B)$
	via
	\begin{equation}
		\label{eq:tilde_v}
		\tilde v(x)
		\coloneqq
		\int_a^b v(x,y) \d\lambda^1(y)
		.
	\end{equation}

	For almost all $x \in B$, we can invoke \cref{lem:one-dimensional}
	and obtain
	\begin{equation*}
		\norm{v}_{L^\infty(B \times I)}
		\int_a^b \abs{ \bar\varphi(x,y) v(x,y) } \d\lambda^1(y)
		\ge
		c
		\parens*{
			\int_a^b v(x,y) \d\lambda^1(y)
		}^2
		=
		c
		\tilde v(x)^2
		.
	\end{equation*}
	By integration, we get
	\begin{equation}
		\label{eq:ineq_tilde_v}
		\norm{v}_{L^\infty(B \times I)}
		\int_{B \times I} \abs{ \bar\varphi v } \d\lambda^d
		\ge
		c
		\norm{\tilde v}_{L^2(B)}^2
		.
	\end{equation}
	Next, we identify $v$ with the measure on $\R^d$
	which has $v$ as a density w.r.t.\ the Lebesgue measure $\lambda^d$.
	Similarly, we view $\tilde v$ as a measure on $\R^{d-1}$.
	An application of \eqref{eq:equivalent_dual_norms}
	yields
	\begin{equation*}
		c \norm{v}_{L^{\alpha,p}(\R^d)\dualspace}^{p'}
		\le
		\int_{B \times I}
		\int_0^1
		\parens*{
			\frac{ \int_{B_r(x,y)} v \d\lambda^d }{r^{d - \alpha p}}
		}^{p' - 1}
		\frac{\d\lambda^1(r)}{r}
		v(x,y) \d\lambda^d(x,y)
		.
	\end{equation*}
	Here, $B_r(x,y)$ is a ball in $\R^d$.
	Due to
	\begin{equation*}
		B_r(x,y) \cap (B \times I) \subset B_r(x) \times I,
	\end{equation*}
	where $B_r(x)$ is a ball in $\R^{d-1}$,
	and since $v$ is supported on $B \times I$,
	we have
	\begin{equation*}
		\int_{B_r(x,y)} v \d\lambda^d
		\le
		\int_{B_r(x) \times I} v \d\lambda^d
		=
		\int_{B_r(x)} \tilde v \d\lambda^{d-1}
		.
	\end{equation*}
	We insert this above and get
	\begin{equation*}
		c \norm{v}_{L^{\alpha,p}(\R^d)\dualspace}^{p'}
		\le
		\int_{B \times I}
		\int_0^1
		\parens*{
			\frac{ \int_{B_r(x)} \tilde v \d\lambda^{d-1} }{r^{d - \alpha p}}
		}^{p' - 1}
		\frac{\d\lambda^1(r)}{r}
		v(x,y) \d\lambda^d(x,y)
		.
	\end{equation*}
	Since the first part of the integrand no longer depends on $y$,
	we can continue with
	\begin{align*}
		c \norm{v}_{L^{\alpha,p}(\R^d)\dualspace}^{p'}
		&\le
		\int_{B}
		\int_0^1
		\parens*{
			\frac{ \int_{B_r(x)} \tilde v \d\lambda^{d-1} }{r^{d - \alpha p}}
		}^{p' - 1}
		\frac{\d\lambda^1(r)}{r}
		\tilde v(x) \d\lambda^{d-1}(x)
		\\
		&=
		\int_{B}
		\int_0^1
		\parens*{
			\frac{ \int_{B_r(x)} \tilde v \d\lambda^{d-1} }{r^{d - 1 - (\alpha - 1/p) p}}
		}^{p' - 1}
		\frac{\d\lambda^1(r)}{r}
		\tilde v(x) \d\lambda^{d-1}(x)
		.
	\end{align*}
	Note that $\alpha - 1/p > 0$ by assumption.
	Thus, we can employ
	\eqref{eq:equivalent_dual_norms} with $d$ and $\alpha$ replaced by $d - 1$ and $\alpha - 1/p$,
	respectively,
	to obtain
	\begin{equation*}
		c \norm{v}_{L^{\alpha,p}(\R^d)\dualspace}^{p'}
		\le
		C \norm{\tilde v}_{L^{\alpha-1/p,p}(\R^{d-1})\dualspace}^{p'}
		.
	\end{equation*}
	We have the Sobolev-type embedding
	\begin{equation*}
		L^{\alpha-1/p,p}(\R^{d-1})
		\embeds
		L^q(\R^{d-1})
	\end{equation*}
	for $q = (d-1) p / ( (d-1) - (\alpha - 1/p) p) = (d-1) p / (d - \alpha p)$
	in case $\alpha p < d$ and some arbitrary $q \in [2,\infty)$ for $\alpha p = d$,
	see \cite[2.8.4~Theorem and 2.8.6~Remark]{Ziemer1989}.
	The assumed lower bound on $p$ guarantees that $q \ge 2$.
	Consequently,
	the boundedness of $B \subset \R^{d-1}$ shows
	\begin{equation*}
		L^{\alpha-1/p,p}(\R^{d-1})
		\embeds
		L^q(\R^{d-1})
		\embeds
		L^2(B)
		.
	\end{equation*}
	Dualizing this embedding yields
	\begin{equation*}
		\norm{\tilde v}_{L^{\alpha-1/p,p}(\R^{d-1})\dualspace}
		\le
		C \norm{\tilde v}_{L^2(B)}
		.
	\end{equation*}
	Thus,
	\begin{equation*}
		c \norm{v}_{L^{\alpha,p}(\R^d)\dualspace}^{2}
		\le
		\norm{\tilde v}_{L^{\alpha-1/p,p}(\R^{d-1})\dualspace}^{2}
		\le
		C \norm{\tilde v}_{L^2(B)}^{2}
		\le
		C^2 \int_{B \times I} \abs{\bar\varphi v} \d\lambda^d
		\norm{v}_{L^\infty(B \times I)}
	\end{equation*}
	finishes the local step.
	We summarize its findings.
	Given a point $z \in \set{\bar\varphi = 0} \cap \bar\Omega$
	with $\partial_d \bar\varphi(z) > 0$,
	we find an open neighborhood  $W_z \subset \R^d$ of $z$
	such that all for all $v \in L^\infty(W_z)$ with $v \ge 0$
	we have
	\begin{equation*}
		\norm{v}_{L^{\alpha,p}(\R^d)\dualspace}^{2}
		\le
		C \int_{W_z} \abs{\bar\varphi v} \d\lambda^d
		\norm{v}_{L^\infty(W_z)}
		.
	\end{equation*}
	By working in rotated coordinate systems, 
	the same conclusion holds for all $z \in \set{\bar\varphi = 0} \cap \bar\Omega$,
	since we have assumed $\nabla\bar\varphi \ne 0$
	on this set.

	Now, we use a partition-of-unity argument to finish the proof.
	Due to compactness of $\bar\Omega$,
	we can find finitely many $z_1, \ldots, z_n \in \set{\bar\varphi = 0} \cap \bar\Omega$,
	such that
	\begin{equation*}
		\set{\bar\varphi = 0} \cap \bar\Omega
		\subset
		\bigcup_{k = 1}^n W_{z_k}
		\eqqcolon
		W
	\end{equation*}
	holds.
	This implies the existence of $\varepsilon > 0$ such that
	$\abs{\bar\varphi} \ge \varepsilon$ on the compact set $\bar\Omega \setminus W$.
	We define $W_0 \coloneqq \set{ \abs{\bar\varphi} > \varepsilon/2}$.
	Together with Hölder's inequality, we get
	\begin{equation*}
		\norm{v}_{L^2(\Omega)}^{2}
		\le
		C \int_{W_0} \abs{\bar\varphi v} \d\lambda^d
		\norm{v}_{L^\infty(W_0)}
	\end{equation*}
	for all $v \in L^\infty(W_0)$.
	Together with the Sobolev embedding
	$L^{\alpha,p}(\R^d) \embeds L^2(\R^d)$,
	this implies
	\begin{equation*}
		\norm{v}_{L^{\alpha,p}(\R^d)\dualspace}^{2}
		\le
		C \int_{W_0} \abs{\bar\varphi v} \d\lambda^d
		\norm{v}_{L^\infty(W_0)}
		.
	\end{equation*}

	Since
	$\bigcup_{k = 0}^n W_k$ is an open covering of $\bar\Omega$,
	there exist $\psi_k \in C_c(W_k; [0,1])$, $k = 0,\ldots, n$,
	such that $\sum_{k = 0}^n \psi_k = 1$ on $\Omega$.
	For an arbitrary $v \in L^\infty(\Omega)$ with $v \ge 0$, we now get
	\begin{align*}
		c \norm{v}_{L^{\alpha,p}(\R^d)\dualspace}^2
		&\le
		\sum_{k = 0}^n
		\norm{\psi_k v}_{L^{\alpha,p}(\R^d)\dualspace}^2
		\\&
		\le
		C
		\sum_{k = 0}^n
		\int_{W_k} \abs{\bar\varphi \psi_k v} \d\lambda^d
		\norm{v}_{L^\infty(W_k)}
		\le
		C
		\int_{\Omega} \abs{\bar\varphi v} \d\lambda^d
		\norm{v}_{L^\infty(\Omega)}
		.
	\end{align*}
	This finishes the proof.
\end{proof}
In \cref{thm:main_inequality}, the restriction
$p \le d/\alpha$ was needed
in order to apply \cref{thm:norm_in_dual_space}.
In the case $p > d/\alpha$,
we have
$g_\alpha \in L^{p'}(\R^d)$
for the conjugate exponent $p'$ of $p$.
Consequently,
\cite[Proposition~8.8]{Folland1999}
implies
$g_\alpha \cnv f \in C_0(\R^d)$
and we get the Sobolev embedding
$L^{\alpha,p}(\R^d) \embeds C_0(\R^d)$.
By dualizing this embedding,
we get
$C_0(\R^d)\dualspace = \MM(\R^d) \embeds L^{\alpha,p}(\R^d)\dualspace$.
Consequently,
\eqref{eq:growth_L1}
implies that \eqref{eq:main_inequality_bessel}
also holds in case $p > d/\alpha$.

We give some further comments concerning the restrictions
on $\alpha$ and $p$ in \cref{thm:main_inequality}.

First, we mention that the assumptions
can only be met in case $d > 1$,
since they imply $1 < \alpha p \le d$.
This is not a restriction since \eqref{eq:growth_L1}
is typically enough in dimension $d = 1$.

If one wants to work in the Hilbert space case,
one can always choose $p = 2$ and $\alpha > \frac12$.

Also the choice $\alpha = 1$ and $p = 2 d / (d + 1) < 2$ is possible.
It would be nice to see a proof of the corresponding inequality in \cref{thm:main_inequality}
without the use of the Bessel potential spaces and of \cref{thm:norm_in_dual_space}.
For this, it seems to be sufficient to show
\begin{align*}
	\norm{v}_{L^1(\Omega)}
	&\le
	\norm{\bar\varphi}_{L^1(\Omega)}
	+
	\int_\Omega \bar u (v - \bar\varphi) \d\lambda^d
	+
	\frac{C}{2} \norm{ v - \bar\varphi }_{W_0^{1,p}(\Omega)}^2
	\\&=
	\int_\Omega \bar u v \d\lambda^d
	+
	\frac{C}{2} \norm{ v - \bar\varphi }_{W_0^{1,p}(\Omega)}^2
\end{align*}
for all $v \in W_0^{1,p}(\Omega)$ in a neighborhood of $\bar\varphi$,
cf.\ \cite[Lemma~4.9]{WachsmuthWachsmuth2021}.
For this, it would be sufficient
to prove
\begin{equation*}
	\norm{
		\sign(v) - \sign(\bar\varphi)
	}_{W_0^{1,p}(\Omega)\dualspace}
	\le
	C \norm{ v - \bar\varphi }_{W_0^{1,p}(\Omega)}
\end{equation*}
locally.
One possible route to prove such inequalities would be
a strengthening of the results in
\cite{WachsmuthWachsmuth2025:1}.

\section{Application to optimal control}
\label{sec:application}

In this section,
we want to apply \cref{thm:main_inequality}
in order to prove second-order conditions for the control problem
\eqref{eq:P}.
To this end,
we review the derivation of second-order conditions via weak-$\star$ second subderivatives
in \cref{subsec:abstract_ssc}.
\cref{subsec:subderivatives}
is dedicated to the
computation of these subderivatives in dual Bessel spaces.
Finally, the second-order conditions for a bang-bang optimal control problem
are presented in \cref{subsec:optimal_control}.

\subsection{Abstract second-order conditions via subderivatives}
\label{subsec:abstract_ssc}
We start by recalling results concerning second-order optimality conditions
utilizing weak-$\star$ second subderivatives.
In infinite-dimensional spaces, this has been developed in the works
\cite{ChristofWachsmuth2017:1,WachsmuthWachsmuth2021,BorchardWachsmuth2023}.
We consider the optimization problem
\begin{equation}
	\label{eq:abstract_opt}
	\text{Minimize} \qquad F(x) + G(x)
	.
\end{equation}
We fix the assumptions.
\begin{assumption}
	\label{asm:abstract}
	We assume that the following is satisfied by the data in \eqref{eq:abstract_opt}.
	\begin{enumerate}
		\item
			\label{asm:abstract:1}
			$Y$ is a separable (real) Banach space with (topological) dual space $X$.
		\item
			\label{asm:abstract:2}
			$G \colon X \to (-\infty,\infty]$ and
			$F \colon \dom G \to \R$ are given functions,
			where we use $\dom G \coloneqq \set{ x \in X \given G(x) < \infty}$.
		\item
			\label{asm:abstract:3}
			At the point $\bar x \in \dom G$,
			there exist $F'(\bar x) \in Y$
			and a bounded bilinear form $F''(\bar x) \colon X \times X \to \R$
			such that
			\begin{equation}
				\label{eq:taylor_F}
				\lim_{k \to \infty}
				\frac{
					F(\bar x + t_k h_k)
					-
					F(\bar x)
					-
					t_k \dual{F'(\bar x)}{h_k}
					-
					\frac12 t_k^2 F''(\bar x) h_k^2
				}{
					t_k^2
				}
				=
				0
			\end{equation}
			holds for all sequences
			$\seq{t_k} \subset (0,\infty)$
			and
			$\seq{h_k} \subset X$
			such that
			$t_k \to 0$, $h_k \weaklystar h \in X$,
			and
			$\bar x + t_k h_k \in \dom G$.
		\item
			\label{asm:abstract:4}
			The map $h \mapsto F''(\bar x) h^2$
			is sequentially weak-$\star$
			continuous.
	\end{enumerate}
\end{assumption}
Here, we used the usual abbreviation
$F''(\bar x) h_k^2 \coloneqq F''(\bar x)[h_k, h_k]$.
We remark that the above mentioned contributions
also contain results which do not require the sequential weak-$\star$ continuity
of $h \mapsto F''(\bar x)h^2$.
However, for the problem under investigation, this assumption is satisfied
and, therefore, we restrict ourself to the situation that this regularity holds.

Note that we do not assume any form of smoothness of the functional $G$.
Instead, the curvature of $G$ will be measured by the following object.
\begin{definition}
	\label{def:second_subderivative}
	The weak-$\star$ second subderivative of $G$ at $\bar x \in \dom G$
	for $\varphi \in Y$ is denoted by
	$G''_X(\bar x, \varphi; \cdot) \colon X \to [-\infty,\infty]$
	and defined via
	\begin{equation*}
		G''_X(\bar x, \varphi; h)
		=
		\inf\set*{
			\liminf_{k \to \infty} \frac{G(\bar x + t_k h_k) - G(\bar x) - t_k \dual{\varphi}{h_k}}{t_k^2/2}
			\given
			t_k \to 0,
			h_k \weaklystar h \text{ in } X
		}
	\end{equation*}
	for all $h \in X$.
	To be precise, the infimum is taken over all sequences
	$\seq{t_k} \subset (0,\infty)$ and $\seq{h_k} \subset X$
	which satisfy $t_k \to 0$ and $h_k \weaklystar h$ in $X$.
\end{definition}
We also provide the following regularity condition
which ensures the existence of recovery sequences with certain properties.
\begin{definition}
	\label{def:epi-differentiable}
	We say that the functional $G$ is
	weak-$\star$ (strictly, strongly, respectively)
	second-order epi-differentiable
	at the point $\bar x \in \dom G$ for $\varphi \in Y$
	in direction $h \in X$,
	if for all $\seq{t_k} \subset (0,\infty)$
	with $t_k \to 0$, there exists a sequence $\seq{h_k} \subset X$
	such that
	$h_k \weaklystar h$
	($h_k \weaklystar h$ and $\norm{h_k}_X \to \norm{h}_X$, $h_k \to h$, respectively)
	and
	\begin{equation*}
		G''_X( \bar x, \varphi; h)
		=
		\lim_{k \to \infty}
		\frac{G(\bar x + t_k h_k) - G(\bar x) - t_k \dual{\varphi}{h_k}}{t_k^2/2}
		.
	\end{equation*}
\end{definition}
We mention that the second-order epi-differentiability
is not needed for the derivation of optimality conditions,
but plays a crucial role for the sensitivity analysis,
i.e., if one wants to differentiate the solution of an optimization problem
w.r.t.\ some parameters,
see \cite{ChristofWachsmuth2017:3}.

The main result concerning second-order conditions using second subderivatives
is the following one.
\begin{theorem}
	\label{thm:second-order_condition}
	Let \cref{asm:abstract} be satisfied.
	A quadratic growth condition for \eqref{eq:abstract_opt} at $\bar x$ holds,
	i.e., there exist $c, \varepsilon > 0$
	such that
	\begin{equation}
		\label{eq:QG}
		F(x) + G(x)
		\ge
		F(\bar x) + G(\bar x)
		+
		\frac{c}{2} \norm{x - \bar x}_X^2
		\qquad
		\forall x \in B^X_\varepsilon(\bar x),
	\end{equation}
	if and only if
	the second-order condition
	\begin{equation}
		\label{eq:ssc}
		F''(\bar x) h^2
		+
		G''_X(\bar x, -F'(\bar x); h)
		>
		0
		\qquad\forall h \in X \setminus \set{0}
	\end{equation}
	and the non-degeneracy condition (NDC)
	\begin{equation}
		\label{eq:NDC}
		\begin{aligned}
			&\text{for all $\seq{t_k} \subset (0,\infty)$, $\seq{h_k} \subset X$
				with $t_k \to 0$, $h_k \weaklystar 0$
			and $\norm{h_k}_X = 1$, we have}\\
			&\qquad
			\liminf_{k \to \infty} \parens[\bigg]{
				\frac1{t_k^2} \parens[\big]{G(\bar x + t_k h_k) - G(\bar x)}
				+
				\dual{F'(\bar x)}{h_k / t_k}
				+
				\frac12 F''(\bar x) h_k^2
			}
			> 0
		\end{aligned}
	\end{equation}
	hold.
\end{theorem}
For the proof, we refer to \cite[Theorem~2.20]{BorchardWachsmuth2023}.
In \eqref{eq:QG},
$B^X_\varepsilon(\bar x)$ denotes a closed ball in the space $X$.
We also mention that \eqref{eq:ssc}
with ``$>$'' replaced by ``$\ge$'' is a necessary optimality condition.

Under \itemref{asm:abstract:4}, it is straightforward to check that the NDC \eqref{eq:NDC}
is equivalent to
\begin{equation}
	\label{eq:NDC2}
	\begin{aligned}
		&\text{for all $\seq{t_k} \subset (0,\infty)$, $\seq{h_k} \subset X$
			with $t_k \to 0$, $h_k \weaklystar 0$
		and $\norm{h_k}_X = 1$, we have}\\
		&\qquad
		\liminf_{k \to \infty} \parens[\bigg]{
			\frac1{t_k^2} \parens[\big]{G(\bar x + t_k h_k) - G(\bar x)}
			+
			\dual{F'(\bar x)}{h_k / t_k}
		}
		> 0
		.
	\end{aligned}
\end{equation}

Finally, we mention a key result for the verification of NDC,
cf.\ \cite[Lemma~2.12]{WachsmuthWachsmuth2021}.
\begin{lemma}
	\label{lem:verification_NDC}
	Let \cref{asm:abstract} be satisfied.
	We assume that there exist constants $c, \varepsilon > 0$ such that
	\begin{equation}
		\label{eq:growth_for_NDC}
		G(x) - G(\bar x) + \dual{F'(\bar x)}{x - \bar x}
		\ge
		\frac{c}{2} \norm{x - \bar x}_X^2
		\qquad\forall
		x \in B^X_\varepsilon(\bar x)
	\end{equation}
	holds.
	Then, \eqref{eq:NDC} and \eqref{eq:NDC2} are satisfied.
\end{lemma}
The proof is straightforward.

\subsection{Second subderivatives in dual Bessel potential spaces}
\label{subsec:subderivatives}
In the works
\cite{ChristofWachsmuth2017:1,WachsmuthWachsmuth2021},
the abstract theory from \cref{subsec:abstract_ssc}
has been successfully applied
to optimal control problems with bang-bang controls.
Therein, the starting point is the growth inequality
\eqref{eq:main_inequality}.
Indeed, in the setting of a box-constrained optimal control problem,
the functional $G$ is the indicator function of the admissible set $\Uad$
and the adjoint state $\bar\varphi$ coincides with the derivative $F'(\bar u)$.
Consequently,
the inequality \eqref{eq:main_inequality} coincides with \eqref{eq:growth_for_NDC}
for the choice $X = L^1(\Omega)$.
However, since $L^1(\Omega)$ is not a dual space,
one uses the space of measures $X = \MM(\Omega) = C_0(\Omega)\dualspace$.
Here, the isometric embedding of $L^1(\Omega)$ in $\MM(\Omega)$ is crucial.

This choice of $X$ has a significant disadvantage.
Indeed, in order to satisfy \itemref{asm:abstract:3},
we basically need that the linearized control-to-state operator
maps from the space $X$ to $L^2(\Omega)$.
In case of a semilinear and elliptic PDE,
this limits the dimension of the underlying domain to $d \le 3$.
Some more details will be given in \cref{subsec:optimal_control}.

In \cref{thm:main_inequality},
we have proved the growth inequality \eqref{eq:main_inequality_bessel}
and, therefore,
we can now invoke the theory of \cref{subsec:abstract_ssc}
in the dual space $L_0^{\alpha,p}(\Omega)\dualspace$.
The actual computation of the weak-$\star$ second subderivatives in
\cite{ChristofWachsmuth2017:1,WachsmuthWachsmuth2021}
is very laborious
and one might fear that complicated arguments are needed
to compute the subderivatives in the space $L_0^{\alpha,p}(\Omega)\dualspace$.
As we will see in \cref{thm:equal_subderivatives} below,
this is actually not the case
and we can benefit from the already known subderivative in the space $\MM(\Omega)$.

In what follows,
we consider two settings.
The first one is the traditional choice
\begin{equation*}
	Y = C_0(\Omega),
	\qquad
	X = C_0(\Omega)\dualspace = \MM(\Omega)
\end{equation*}
and the second one is
\begin{equation*}
	Y = L_0^{\alpha,p}(\Omega),
	\qquad
	X = L_0^{\alpha,p}(\Omega)\dualspace
	.
\end{equation*}
In both cases,
the space $L^\infty(\Omega)$ is continuously embedded in $X = Y\dualspace$
where the embedding is defined by duality via
\begin{equation*}
	\dual{u}{\varphi}
	\coloneqq
	\int_\Omega u \varphi \d\lambda^d
	\qquad
	\forall u \in L^\infty(\Omega), \varphi \in Y.
\end{equation*}
Consequently, the feasible set $\Uad$ of \eqref{eq:P}
can be understood as a subset of $X$.
We consider its indicator function, i.e.,
the functional $G \colon X \to [0,\infty]$,
\begin{equation*}
	G(u)
	\coloneqq
	\begin{cases}
		0 & \text{if } u \in L^\infty(\Omega) \text{ and } \abs{u} \le 1 \text{ a.e.}, \\
		+\infty & \text{else}.
	\end{cases}
\end{equation*}
We investigate the weak-$\star$ second subderivative of $G$
from \cref{def:second_subderivative}
for the above choices of $X$.

For the remaining part of this section,
let $\bar u \in \Uad$ and $\bar\varphi \in C^1(\R^d)$ be given
such that
\cref{asm:standing}
and \eqref{eq:FONC} hold.
Further, we assume that $\bar\varphi |_{\Omega} \in C_0(\Omega)$.
Note that \eqref{eq:FONC} is equivalent to $\bar\varphi \in \partial G(\bar u)$,
where $\partial G$ denotes the convex subdifferential of $G$,
and to $\bar u = \sign(\bar\varphi)$.
This directly implies $G''_X(\bar u, \bar\varphi; h) \ge 0$
for all $h \in X$,
see \cite[Lemma~2.4]{WachsmuthWachsmuth2021}.

From
\cite[Theorem~6.11]{ChristofWachsmuth2017:1}
or
\cite[Theorem~4.13]{WachsmuthWachsmuth2021},
we get
\begin{equation*}
	G''_{\MM(\Omega)}(\bar u, \bar\varphi; h)
	=
	\begin{cases}
		\displaystyle\int_{\set{\bar\varphi = 0} \cap \Omega}
		\frac{\abs{\nabla\bar\varphi}}{2}
		\parens*{
			\frac{\d h}{\d\HH^{d-1} |_{\set{\bar\varphi = 0} \cap \Omega} }
		}^2
		\d\HH^{d-1}
		& \text{if } h \ll \HH^{d-1} |_{\set{\bar\varphi = 0} \cap \Omega }
		,
		\\
		+\infty & \text{else}
	\end{cases}
\end{equation*}
for all $h \in \MM(\Omega)$.
Here,
$h \ll \HH^{d-1} |_{\set{\bar\varphi = 0} \cap \Omega }$
means that
$h$ is absolutely continuous w.r.t.\ $\HH^{d-1} |_{\set{\bar\varphi = 0} \cap \Omega }$
and
$\d h / \d\HH^{d-1} |_{\set{\bar\varphi = 0} \cap \Omega}$
denotes the corresponding Radon--Nikodým derivative.
By \cref{asm:standing}, $\abs{\nabla\varphi}$
is uniformly positive and bounded on $\set{\bar\varphi = 0} \cap \Omega$.
Consequently, the formula for the subderivative can be written as
\begin{equation*}
	G''_{\MM(\Omega)}(\bar u, \bar\varphi; h)
	=
	\begin{cases}
		\displaystyle\int_{\set{\bar\varphi = 0} \cap \Omega}
		\frac{\abs{\nabla\bar\varphi}}{2}
		h^2
		\d\HH^{d-1}
		& \text{if } h \in L^2(\HH^{d-1} | _{\set{\bar\varphi = 0} \cap \Omega })
		,
		\\
		+\infty & \text{else},
	\end{cases}
\end{equation*}
where we identify
$L^2(\HH^{d-1} | _{\set{\bar\varphi = 0} \cap \Omega })$
with a subset of $\MM(\Omega)$ via the map
$\dual{h}{\varphi} \coloneqq \int_{\set{\bar\varphi = 0} \cap \Omega} h \varphi \d\HH^{d-1}$
for $h \in L^2(\HH^{d-1} | _{\set{\bar\varphi = 0} \cap \Omega })$
and $\varphi \in C_0(\Omega)$.
Moreover,
the above references also provide us with the strict second-order epi-differentiability
of $G$ in the space $X = \MM(\Omega)$
at $\bar u$ for $\bar\varphi$ for all directions $h \in \MM(\Omega)$.

In order to apply the abstract theory with $Y = L_0^{\alpha,p}(\Omega)$,
we need to ensure that $\bar\varphi$ belongs to this space.
From $\bar\varphi \in C^1(\R^d)$ and $\bar\varphi|_\Omega \in C_0(\Omega)$,
we get $\bar\varphi|_\Omega \in W_0^{1,p}(\Omega)$.
Consequently, we get
$\bar\varphi|_\Omega \in W_0^{1,p}(\Omega) = L_0^{1,p}(\Omega) \embeds L_0^{\alpha,p}(\Omega)$
for all $\alpha \in (0,1]$.
With a small abuse of notation, we do not distinguish
between $\bar\varphi$ and $\bar\varphi|_\Omega \in L_0^{\alpha,p}(\Omega)$.

Further, we mention that the assumptions on $\alpha$ and $p$ from \cref{thm:main_inequality}
gives rise to a trace operator from
$L_0^{\alpha,p}(\Omega)$
to
$L^2( \HH^{d-1} |_{\set{\bar\varphi = 0} \cap \Omega})$,
cf.\ \cite{Schneider2011}
and \cite[Corollary~8.3.2]{Schneider2021}.
Consequently,
we can identify
$L^2( \HH^{d-1} |_{\set{\bar\varphi = 0} \cap \Omega})$
with a subspace of $L_0^{\alpha,p}(\Omega)\dualspace$
via
$\dual{h}{\varphi} \coloneqq \int_{\set{\bar\varphi = 0} \cap \Omega} h \varphi \d\HH^{d-1}$
for all
$h \in L^2( \HH^{d-1} |_{\set{\bar\varphi = 0} \cap \Omega})$
and
$\varphi \in L_0^{\alpha,p}(\Omega)$.
\begin{theorem}
	\label{thm:equal_subderivatives}
	Let \cref{asm:standing} be satisfied by $\Omega$ and $\bar\varphi$.
	Additionally, we assume $\bar\varphi|_\Omega \in C_0(\Omega)$.
	We set $\bar u = \sign(\bar\varphi)$.
	Further, let $\alpha$, $p$ satisfy the requirements from \cref{thm:main_inequality}
	and, additionally, $\alpha \le 1$.
	Then,
	\begin{equation*}
		G''_{L_0^{\alpha,p}(\Omega)\dualspace}(\bar u, \bar\varphi; h)
		=
		\begin{cases}
			\displaystyle\int_{\set{\bar\varphi = 0} \cap \Omega}
			\frac{\abs{\nabla\bar\varphi}}{2}
			h^2
			\d\HH^{d-1}
			& \text{if } h \in L^2( \HH^{d-1} |_{\set{\bar\varphi = 0} \cap \Omega}),
			\\
			+\infty & \text{otherwise}.
		\end{cases}
	\end{equation*}
	for all $h \in L_0^{\alpha,p}(\Omega)\dualspace$.
	Here,
	we identify
	$L^2( \HH^{d-1} |_{\set{\bar\varphi = 0} \cap \Omega})$
	with a subspace of $L_0^{\alpha,p}(\Omega)\dualspace$
	as mentioned above.

	Moreover,
	the functional $G$ is weak-$\star$
	second-order epi-differentiable
	in the space $X = L_0^{\alpha,p}(\Omega)\dualspace$
	at $\bar u$ for $\bar\varphi$
	in all directions $h \in L_0^{\alpha,p}(\Omega)\dualspace$.
\end{theorem}
\begin{proof}
	We show
	\begin{subequations}
		\label{eq:partial_statement_subderivative}
		\begin{align}
			\label{eq:partial_statement_subderivative:1}
			G''_{L_0^{\alpha,p}(\Omega)\dualspace}(\bar u, \bar\varphi; h)
			&=
			G''_{\MM(\Omega)}(\bar u, \bar\varphi; h)
			&&
			\forall h \in L^2( \HH^{d-1} |_{\set{\bar\varphi = 0} \cap \Omega})
			,
			\\
			\label{eq:partial_statement_subderivative:2}
			G''_{L_0^{\alpha,p}(\Omega)\dualspace}(\bar u, \bar\varphi; h)
			&=
			+\infty
			&&
			\forall h \in L_0^{\alpha,p}(\Omega)\dualspace \setminus L^2( \HH^{d-1} |_{\set{\bar\varphi = 0} \cap \Omega})
			.
		\end{align}
	\end{subequations}
	Due to the formula for $G''_{\MM(\Omega)}(\bar u, \bar\varphi; h)$
	from
	\cite[Theorem~6.11]{ChristofWachsmuth2017:1}
	or
	\cite[Theorem~4.13]{WachsmuthWachsmuth2021},
	this is sufficient to verify the formula for the weak-$\star$ second subderivative.

	Let a direction $h \in L_0^{\alpha,p}(\Omega)\dualspace$
	with $G''_{L_0^{\alpha,p}(\Omega)\dualspace}(\bar u, \bar \varphi; h) < \infty$ be given.
	By definition of the second subderivative,
	for every $\varepsilon > 0$, we get the existence of sequences
	$t_k \to 0$, $h_k \weaklystar h$
	such that
	\begin{equation*}
		G''_{L_0^{\alpha,p}(\Omega)\dualspace}(\bar u, \bar \varphi; h)
		+
		\varepsilon
		\ge
		\lim_{k \to \infty} \frac{G(\bar u + t_k h_k) - G(\bar u) - t_k \dual{\bar \varphi}{h_k}}{t_k^2/2}
		.
	\end{equation*}
	This shows that the sequence $\set{\bar u + t_k h_k}$ belongs to $\Uad$.
	In particular,
	$\seq{h_k} \subset L^\infty(\Omega)$,
	and
	\begin{equation*}
		G''_{L_0^{\alpha,p}(\Omega)\dualspace}(\bar u, \bar \varphi; h)
		+
		\varepsilon
		\ge
		\lim_{k \to \infty} \frac{- t_k \dual{\bar \varphi}{h_k}}{t_k^2/2}
		.
	\end{equation*}
	Combining $\bar u + t_k h_k \in \Uad$
	with $\bar u = \sign(\bar\varphi)$,
	we get $\bar\varphi h_k \le 0$ a.e.\ in $\Omega$,
	i.e.,
	\begin{equation*}
		G''_{L_0^{\alpha,p}(\Omega)\dualspace}(\bar u, \bar \varphi; h)
		+
		\varepsilon
		\ge
		\lim_{k \to \infty} \frac{2}{t_k^2} \int_\Omega \abs{\bar\varphi t_k h_k} \d\lambda^d
		.
	\end{equation*}
	Due to $\norm{t_k h_k}_{L^\infty(\Omega)} \le 2$,
	we can apply \eqref{eq:growth_L1}
	and this shows
	\begin{equation*}
		\infty
		>
		G''_{L_0^{\alpha,p}(\Omega)\dualspace}(\bar u, \bar \varphi; h)
		+
		\varepsilon
		\ge
		\lim_{k \to \infty} \frac{2 c}{t_k^2} \norm{t_k h_k}_{L^1(\Omega)}^2
		.
	\end{equation*}
	Hence, the sequence $\seq{h_k}$ is bounded in $L^1(\Omega)$ and w.l.o.g.\ we can assume
	$h_k \weaklystar \tilde h$ in $\MM(\Omega)$ for some $\tilde h \in \MM(\Omega)$.
	For all test functions $\psi \in C_c^\infty(\Omega)$,
	we have
	\begin{equation*}
		\dual{\tilde h}{\psi}_{C_0(\Omega)}
		=
		\lim_{k \to \infty} \dual{h_k}{\psi}_{C_0(\Omega)}
		=
		\lim_{k \to \infty} \dual{h_k}{\psi}_{L_0^{\alpha,p}(\Omega)}
		=
		\dual{h}{\psi}_{L_0^{\alpha,p}(\Omega)}
		.
	\end{equation*}
	Since $C_c^\infty(\Omega)$ is dense in $C_0(\Omega)$ and $L_0^{\alpha,p}(\Omega)$,
	this shows $\tilde h = h$.
	Further,
	\begin{equation*}
		\infty
		>
		G''_{L_0^{\alpha,p}(\Omega)\dualspace}(\bar u, \bar \varphi; h)
		+
		\varepsilon
		\ge
		\lim_{k \to \infty} \frac{2}{t_k^2} \int_\Omega \abs{\bar\varphi t_k h_k} \d\lambda^d
		\ge
		G''_{\MM(\Omega)}(\bar u, \bar \varphi; h)
	\end{equation*}
	by definition of the subderivative in $\MM(\Omega)$.
	With $\varepsilon \to 0$,
	\begin{equation*}
		G''_{L_0^{\alpha,p}(\Omega)\dualspace}(\bar u, \bar \varphi; h)
		\ge
		G''_{\MM(\Omega)}(\bar u, \bar \varphi; h)
	\end{equation*}
	follows
	for all $h \in L_0^{\alpha,p}(\Omega)$ with
	$G''_{L_0^{\alpha,p}(\Omega)\dualspace}(\bar u, \bar \varphi; h) < \infty$.
	In particular, this shows
	$h \in L^2( \HH^{d-1} |_{\set{\bar\varphi = 0} \cap \Omega})$.
	Consequently,
	\eqref{eq:partial_statement_subderivative:2} follows.

	Now, let $h \in L^2( \HH^{d-1} |_{\set{\bar\varphi = 0} \cap \Omega})$
	be given.
	From the weak-$\star$ twice epi-differentiability in $\MM(\Omega)$,
	for every sequence $\seq{t_k} \subset (0,\infty)$ with $t_k \to 0$,
	we get a sequence $\seq{h_k} \subset \MM(\Omega)$ with $h_k \weaklystar h$ in $\MM(\Omega)$
	such that
	\begin{equation*}
		G''_{\MM(\Omega)}(\bar u, \bar \varphi; h)
		=
		\lim_{k \to \infty} \frac{G(\bar u + t_k h_k) - G(\bar u) - t_k \dual{\bar \varphi}{h_k}}{t_k^2/2}
		.
	\end{equation*}
	Now we can argue as above to show
	\begin{equation*}
		\infty
		>
		G''_{\MM(\Omega)}(\bar u, \bar \varphi; h)
		\ge
		\lim_{k \to \infty} \frac{2 c}{t_k^2} \norm{t_k h_k}_{L^{\alpha,p}(\R^d)\dualspace}^2
		\ge
		\lim_{k \to \infty} \frac{2 c}{t_k^2} \norm{t_k h_k}_{L_0^{\alpha,p}(\Omega)\dualspace}^2
	\end{equation*}
	via \cref{thm:main_inequality}.
	Consequently, $h_k$ is bounded in $L_0^{\alpha,p}(\Omega)\dualspace$
	and $h_k \weaklystar h$ in $L_0^{\alpha,p}(\Omega)\dualspace$ follows.
	The definition of the weak-$\star$ second subderivatives yields
	\begin{equation*}
		G''_{\MM(\Omega)}(\bar u, \bar \varphi; h)
		=
		\lim_{k \to \infty} \frac{G(\bar u + t_k h_k) - G(\bar u) - t_k \dual{\bar \varphi}{h_k}}{t_k^2/2}
		\ge
		G''_{L_0^{\alpha,p}(\Omega)\dualspace}(\bar u, \bar \varphi; h),
	\end{equation*}
	while the reverse inequality has been shown in the first part of the proof.
	Hence, the sequences $\seq{t_k}$ and $\seq{h_k}$ can also serve as recovery sequences
	in the space $L_0^{\alpha,p}(\Omega)\dualspace$.
\end{proof}

\subsection{Second-order conditions for bang-bang control}
\label{subsec:optimal_control}
Finally, we want to use the above findings to study the problem \eqref{eq:P}
with the objective $F$ given in \eqref{eq:obj}
and with the control-to-state map $S$ given by the solution operator
$u \mapsto y$
of the (weak formulation) of the semilinear and elliptic PDE
\begin{equation}
	\label{eq:PDE}
	-\Delta y + a(y) = u \qquad \text{in } \Omega
\end{equation}
equipped with homogeneous Dirichlet conditions.
Here, $\Omega \subset \R^d$ is a bounded Lipschitz domain
with arbitrary dimension $d \ge 2$.
We assume that $a \colon \R \to \R$ is $C^2$
and $a' \ge 0$.
By the usual arguments,
one can check that
\eqref{eq:PDE} has a unique solution
$y \in H_0^1(\Omega) \cap L^\infty(\Omega)$
for every
$u \in L^\infty(\Omega)$,
see \cite[Section~4.2]{Troeltzsch2010:1}
for the case of homogeneous Neumann boundary conditions.
Owing to the celebrated implicit function theorem,
one obtains that $S \colon L^\infty(\Omega) \to H_0^1(\Omega) \cap L^\infty(\Omega)$
is twice continuously Fréchet differentiable
and
$z = S'(u) h$ is the unique solution of the linearized PDE
\begin{equation}
	\label{eq:linearized_state}
	-\Delta z + a'(y) z = h \qquad \text{in } \Omega
\end{equation}
equipped with homogeneous Dirichlet boundary conditions.
Consequently, the functional $F \colon L^\infty(\Omega) \to \R$
from \eqref{eq:obj}
is twice continuously differentiable and one can compute
\begin{align}
	\label{eq:first_derivative_F}
	F'(u) h &= \int_\Omega \varphi h \d\lambda^d, \\
	\label{eq:second_derivative_F}
	F''(u) [h_1, h_2] &=
	\int_\Omega (1 - a''(y) \varphi) [S'(u) h_1] [S'(u) h_2] \d\lambda^d,
\end{align}
where $y = S(u)$ is the state and $\varphi = S'(u) (y - y_d)$ is the adjoint state,
i.e.,
\begin{equation}
	\label{eq:adjoint}
	-\Delta \varphi + a'(y) \varphi = y - y_d \qquad \text{in } \Omega
\end{equation}
equipped with homogeneous Dirichlet conditions.

We are going to apply the theory from \cref{subsec:abstract_ssc}
with the choices
$X = \MM(\Omega)$ and $X = L_0^{\alpha,p}(\Omega)\dualspace$.
In order to verify
\itemref{asm:abstract:3} and \ref{asm:abstract:4},
we have to check that the bilinear form
$F''(\bar u)$ extends to the space $X$.
This is possible if $S'(\bar u)$ can be extended to a continuous mapping
from $X$ to $L^2(\Omega)$.
For the choice $X = \MM(\Omega)$,
it is well known that this only works for dimensions $d \le 3$.
This is precisely the reason for the dimension restriction
in the second-order results
\cite[Example~6.14]{ChristofWachsmuth2017:1}
or
\cite[Section~5]{WachsmuthWachsmuth2021}.

We argue that this dimension restriction can be avoided if we work
in a dual Bessel space $X = L_0^{\alpha,p}(\Omega)\dualspace$.
In order to prove that $S(u)$ can be extended to an operator
from $L_0^{\alpha,p}(\Omega)\dualspace$ to $L^2(\Omega)$,
we use the classical technique from
\cite[Théorème~1]{Stampacchia1960}, see also \cite[Lemma~3.4]{ChristofWachsmuth2021:2}.
\begin{theorem}
	\label{thm:stampacchia_meets_bessel}
	Suppose that $\Omega \subset \R^d$ is a bounded Lipschitz domain, $d \ge 2$.
	Let $\alpha \in [0,1]$ and $p \in (1,\infty)$ be given such that
	\begin{equation}
		\label{eq:assumption_stampacchia}
		\frac12
		<
		\frac1p + \frac{1-\alpha}d
		\le
		1
		.
	\end{equation}
	We define
	\begin{equation*}
		\tau
		\coloneqq
		1 - \frac1p - \frac{2 - \alpha}{d}
		=
		1
		-
		\parens*{
			\frac1p
			+
			\frac{1-\alpha}{d}
		}
		-
		\frac1d
		.
	\end{equation*}
	Further, let $u \in L^\infty(\Omega)$ be arbitrary.
	Then, the operator $S'(u) \colon L^\infty(\Omega) \to L^2(\Omega)$
	given by the solution map of \eqref{eq:linearized_state}
	can be extended uniquely to a continuous operator
	$S'(u) \colon L_0^{\alpha,p}(\Omega)\dualspace \to L^r(\Omega)$,
	where
	we can choose
	$r \in [1, \infty]$
	in case that $\tau < 0$
	and where
	$r \in [1, 1/\tau)$ in case that $\tau \ge 0$.
	The corresponding norm of $S'(u)$ is independent of $u \in L^\infty(\Omega)$.
\end{theorem}
\begin{proof}
	It is sufficient to show that
	\begin{equation*}
		\norm{S'(u) h}_{L^r(\Omega)}
		\le
		C \norm{h}_{L_0^{\alpha,p}(\Omega)\dualspace}
		\qquad
		\forall h \in L^\infty(\Omega).
	\end{equation*}
	Then, $S'(u)$ can be extended by continuity,
	since
	$L^\infty(\Omega)$ is dense in $L_0^{\alpha,p}(\Omega)\dualspace$.

	Let $h \in L^\infty(\Omega)$ be given.
	The function $z \coloneqq S'(u)h \in H_0^1(\Omega)$ satisfies the weak formulation
	\begin{equation*}
		\int_\Omega \nabla z \nabla \psi + a'(y) z \psi \d\lambda^d
		=
		\int_\Omega h \psi \d\lambda^d
		\qquad\forall \psi \in H_0^1(\Omega),
	\end{equation*}
	where $y = S(u)$.

	For $k > 0$, we define the shrinkage
	$z_k \coloneqq  z - \min\set{k, \max\set{ - k, z} } \in H_0^1(\Omega)$
	and the set
	\begin{equation*}
		L(k)
		\coloneqq
		\set{ z_k \ne 0 }
		=
		\set{ \abs{z} > k }
		.
	\end{equation*}
	By testing the weak formulation with $\psi = z_k$,
	we get
	\begin{equation*}
		\int_\Omega \nabla z \cdot \nabla z_k + a'(y) z z_k \d\lambda^d
		=
		\int_\Omega h z_k \d\lambda^d
		.
	\end{equation*}
	Due to $a'(y) \ge 0$, $z z_k \ge 0$ and $\nabla z \cdot \nabla z_k = \abs{\nabla z_k}^2$,
	we get
	\begin{equation*}
		\int_\Omega \abs{\nabla z_k}^2 \d\lambda^d
		\le
		\int_\Omega h z_k \d\lambda^d
		.
	\end{equation*}
	On the left-hand side, we can use the Poincaré inequality and the Sobolev embedding theorem
	to get
	\begin{equation}
		\label{eq:stampacchia_1}
		c \norm{z_k}_{L^q(\Omega)} \norm{z_k}_{H_0^1(\Omega)}
		\le
		c \norm{z_k}_{H_0^1(\Omega)}^2
		\le
		\int_\Omega h z_k \d\lambda^d
		,
	\end{equation}
	for $q \in (2,\infty)$ given by
	\begin{equation*}
		\frac1q
		=
		\frac12
		-
		\frac1d
	\end{equation*}
	for dimensions $d \ge 3$
	and for arbitrary $q \in (2,\infty)$
	for $d = 2$.
	On the right-hand side,
	\begin{equation}
		\label{eq:stampacchia_2}
		\int_\Omega h z_k \d\lambda^d
		\le
		\norm{h}_{L_0^{\alpha,p}(\Omega)\dualspace}
		\norm{z_k}_{L_0^{\alpha,p}(\Omega)}
		.
	\end{equation}
	By the assumption on $\alpha$ and $p$,
	the number $s \in \R$ defined via
	\begin{equation*}
		\frac{1}{p}
		=
		\frac{1}{s} - \frac{1-\alpha}{d}
	\end{equation*}
	satisfies $s \in [1,2)$.
	The Sobolev embedding theorem implies
	$W_0^{1,s}(\Omega) = L_0^{1,s}(\Omega) \embeds L_0^{\alpha,p}(\Omega)$,
	see \cite[Corollary~3.1.5]{AdamsHedberg1996}.
	Consequently,
	Hölder's inequality implies
	\begin{equation*}
		c
		\norm{z_k}_{L_0^{\alpha,p}(\Omega)}
		\le
		\norm{z_k}_{W_0^{1,s}(\Omega)}
		\le
		\lambda^d( L(k) )^{1/s - 1/2}
		\norm{z_k}_{H_0^1(\Omega)}
		.
	\end{equation*}
	By combining this with \eqref{eq:stampacchia_1} and \eqref{eq:stampacchia_2},
	we arrive at
	\begin{equation*}
		\norm{z_k}_{L^q(\Omega)}
		\le
		c
		\norm{h}_{L_0^{\alpha,p}(\Omega)\dualspace}
		\lambda^d( L(k) )^{1/s - 1/2}
		.
	\end{equation*}
	Now, we can continue as in
	\cite[Lemma~3.4]{ChristofWachsmuth2021:2}.
	Indeed, \cite[(3.10)]{ChristofWachsmuth2021:2}
	yields
	\begin{equation*}
		(m - k)
		\lambda^d(L(m))^{1/q}
		\le
		\parens*{
			\int_{L(m)} (\abs{z} - k)^q \d\lambda^d
		}^{1/q}
		\le
		\norm{z_k}_{L^q(\Omega)}
		\qquad
		\forall
		m > k \ge 0
		.
	\end{equation*}
	Together with the previous inequality,
	we get
	\begin{equation*}
		\lambda^d(L(m))
		\le
		c
		(m - k)^{-q}
		\norm{h}_{L_0^{\alpha,p}(\Omega)\dualspace}^q
		\lambda^d( L(k) )^{\sigma}
		\qquad
		\forall
		m > k \ge 0
	\end{equation*}
	with
	\begin{equation*}
		\sigma
		\coloneqq
		q \parens*{ \frac 1s - \frac 12 }
		=
		q \parens*{ \frac 1p + \frac{1-\alpha}{d} - \frac 12 }
		.
	\end{equation*}
	Note that \eqref{eq:assumption_stampacchia} implies $\sigma > 0$.

	In case $\tau > 0$,
	we have
	\begin{equation*}
		\frac{\sigma}{q}
		=
		-\tau + \frac12 - \frac1d
		\le
		-\tau + \frac1q
		<
		\frac1q,
	\end{equation*}
	i.e.,
	$\sigma < 1$.
	Thus, we can invoke
	\cite[Lemme Préliminaire, case 1]{Stampacchia1960}
	and this implies
	\begin{equation*}
		\lambda^d(\set{\abs{z} \ge k})
		=
		\lambda^d(L(k))
		\le
		C
		k^{-\hat r} \norm{h}_{L_0^{\alpha,p}(\Omega)}^{\hat r}
		\qquad
		\forall k > 0
		,
	\end{equation*}
	where
	\begin{equation*}
		\hat r
		\coloneqq
		\frac{q}{1 - \sigma}
		=
		\frac{1}{
			\frac12
			+
			\frac1q
			-
			\frac1p
			-
			\frac{1-\alpha}{d}
		}
		.
	\end{equation*}
	In case $d > 3$, we have $\hat r = 1/\tau$,
	whereas in case $d = 2$,
	$\hat r$ is smaller than $1/\tau$
	but arbitrarily close to $1/\tau$ (by choosing $q < \infty$ big enough).
	Note that 
	this implies that $z$ belongs to the weak Lebesgue space with exponent $\hat r$
	and its quasinorm can be bounded by a constant times
	$\norm{h}_{L_0^{\alpha,p}(\Omega)}$.
	Since $\lambda^d(\Omega) < \infty$,
	a standard argument shows that
	$\norm{z}_{L^r(\Omega)} \le C \norm{h}_{L_0^{\alpha,p}(\Omega)}$
	for any $r < \hat r$.
	This shows the claim in the case $\tau > 0$.

	The case $\tau = 0$ can be handled by making $p$ a tiny bit bigger
	which brings us back to case $\tau > 0$.
	This provides the estimate
	$\norm{z}_{L^r(\Omega)} \le C \norm{h}_{L_0^{\alpha,p}(\Omega)}$
	for all $r \in (1,\infty)$.

	Finally, in the other case $\tau < 0$,
	we can choose $q$ big enough (in case $d = 2$)
	and this guarantees $\sigma > 1$.
	Consequently,
	\cite[Lemme Préliminaire, case 2]{Stampacchia1960}
	implies
	$\norm{z}_{L^\infty(\Omega)} \le C \norm{h}_{L_0^{\alpha,p}(\Omega)}$.
\end{proof}
We briefly mention that the result of \cref{thm:stampacchia_meets_bessel} is (almost) optimal.
Indeed, since the Laplacian is a second-order differential operator,
the regularity of the solution of \eqref{eq:linearized_state} with a right-hand side in
$L_0^{\alpha,p}(\Omega)\dualspace$
(which is sometimes denoted by $L^{-\alpha, p'}(\Omega)$ with the conjugate exponent $p'$ of $p$)
cannot be better than
$L^{2 - \alpha, p'}(\Omega)$.
By the Sobolev embedding theorem, this space embeds into
$L^q(\Omega)$
for
$1/q = 1/p' - (2-\alpha)/d = 1 - 1/p - (2-\alpha)/d = \tau$
(in case that the right-hand side is positive)
and does not embed into better spaces.
This is almost the regularity guaranteed by the previous result.
In \cite{Stampacchia1960},
the Marcinkiewicz interpolation theorem was used to prove optimal estimates.
In the setting of \cref{thm:stampacchia_meets_bessel},
the application of this interpolation theorem in case $\tau > 0$ should imply
the boundedness of $S'(u)$ from $L_0^{\alpha,p}(\Omega)\dualspace$
to $L^{1/\tau}(\Omega)$.

From \cref{thm:stampacchia_meets_bessel} together with the mean-value inequality,
we directly get the Lipschitz estimate
\begin{equation}
	\label{eq:lipschitz}
	\begin{aligned}
		\norm{ S(u_1) - S(u_2) }_{L^r(\Omega)}
		&\le
		\sup_{t \in [0,1]} \norm{S'(u_1 + t (u_2 - u_1))}
		\norm{u_1 - u_2}_{L_0^{\alpha,p}(\Omega)\dualspace}
		\\&
		\le
		C \norm{u_1 - u_2}_{L_0^{\alpha,p}(\Omega)\dualspace}
	\end{aligned}
\end{equation}
for all $u_1, u_2 \in L^\infty(\Omega)$.

For the required sequential weak-$\star$ continuity of $F''(\bar u)$,
see \itemref{asm:abstract:4},
we need some compactness of $S'(\bar u)$.
This follows easily from a compact embedding.

\begin{lemma}
	\label{lem:compact}
	We use the setting of \cref{thm:stampacchia_meets_bessel}
	with the additional requirement that $r < \infty$ if $\alpha = 1$.
	Then,
	the solution map
	$S'(u) \colon L_0^{\alpha,p}(\Omega)\dualspace \to L^r(\Omega)$
	is compact
	for all $u \in L^\infty(\Omega)$.
\end{lemma}
\begin{proof}
	First, we discuss the case $\alpha < 1$.
	One can check that if $\alpha$, $p$ and $r$ satisfy the assumptions of \cref{thm:stampacchia_meets_bessel},
	then they are also satisfied by $\alpha + \varepsilon$, $p$ and $r$
	if $\varepsilon > 0$ is chosen small enough.
	Consequently,
	$S'(u)$ is bounded from
	$L_0^{\alpha+\varepsilon,p}(\Omega)\dualspace$
	to
	$L^r(\Omega)$.
	From \cite[Proposition~3.1]{BellidoCuetoGarciaSaez2025}
	we get the compactness of the embedding from
	$L_0^{\alpha+\varepsilon,p}(\Omega)$
	to
	$L_0^{\alpha,p}(\Omega)$
	for $0 < \alpha < \alpha + \varepsilon < 1$.
	By dualizing this embedding, we get that
	the embedding from
	$L_0^{\alpha,p}(\Omega)\dualspace$
	to
	$L_0^{\alpha+\varepsilon,p}(\Omega)\dualspace$
	is compact.
	The combination of these two results yields the desired compactness of
	$S'(u) \colon L_0^{\alpha,p}(\Omega)\dualspace \to L^r(\Omega)$.

	In the case $\alpha = 1$,
	we have that $L_0^{1,p}(\Omega) = W_0^{1,p}(\Omega)$
	with equivalent norms for all $p \in (1,\infty)$.
	Due to \eqref{eq:assumption_stampacchia},
	we have $p \in (1,2)$.
	From \cref{thm:stampacchia_meets_bessel},
	we get the boundedness of $S'(u)$ from
	$W_0^{1,q}(\Omega)\dualspace$ to $L^{s_1}(\Omega)$
	for $q$ slightly smaller than $p$
	and $s_1 \in (1,\infty)$ such that $1/s_1$
	is slighly bigger than $1 - 1/q - 1/d$ if this quantity is non-negative
	and $s_1 = \infty$ in case that it is negative.
	By Lax-Milgram,
	we get boundedness from
	$W_0^{1,2}(\Omega)\dualspace$ to $W_0^{1,2}(\Omega)$
	and the latter space embeds compactly into
	$L^{s_2}(\Omega)$
	for all $s_2$ such that
	$1/s_2$ is slightly bigger than
	$1 - 1/2 - 1/d$.
	In order to interpolate,
	it is most convenient
	to consider the adjoints $S'(u)\adjoint$
	and
	to identify $W_0^{1,p}(\Omega)$
	with a closed subset of $L^p(\Omega)^{d+1}$ via the map $u \mapsto (u, \nabla u)$.
	Under this identification,
	$S'(u)\adjoint$
	is bounded
	from $L^{s_1'}(\Omega)$ to $L^q(\Omega)^{d+1}$
	and
	compact from
	$L^{s_2'}(\Omega)$ to $L^2(\Omega)^{d+1}$.
	Since $p$ lies strictly between $q$ and $2$,
	we can apply the interpolation result
	\cite[Theorem~4.2.9]{BennettSharpley1988}
	to obtain the compactness of
	$S'(u)\adjoint$
	from
	$L^s(\Omega)$ to $L^p(\Omega)^{d+1}$,
	where $s \in (1,\infty)$ is chosen appropriately.
	Note that this is possible
	(by chosing $q$, $s_1$ and $s_2$ large enough)
	for all $s$
	such that $1/s$ is only slightly bigger than
	$1 - 1/p - 1/d$
	in case that this quantity is non-negative.
	If it is negative, we can obtain arbitrary large $s < \infty$.
	In fact, it is possible to obtain $s = r$.
	This compactness of $S'(u)\adjoint$
	is equivalent to the compactness of
	$S'(u)$
	from
	$W_0^{1,p}(\Omega)\dualspace = L_0^{1,p}(\Omega)\dualspace$
	to
	$L^r(\Omega)$.
\end{proof}

We collect the assumptions which are needed in the sequel.
\begin{assumption}
	\label{asm:everything}
	We consider the problem \eqref{eq:P}
	with the control-to-state map $S$ given by the solution operator of \eqref{eq:PDE}
	and the tracking term objective $F$ from \eqref{eq:obj}.
	\begin{enumerate}
		\item
			\label{asm:everything:1}
			The set $\Omega \subset \R^d$, $d \ge 2$,
			is a nonempty, open, and bounded Lipschitz domain.
		\item
			\label{asm:everything:2}
			The function $a \colon \R \to \R$ is increasing
			and twice continuously differentiable.
			We have $y_d \in L^\infty(\Omega)$.
		\item
			\label{asm:everything:3}
			The feasible point $\bar u \in \Uad$ of \eqref{eq:P}
			is stationary,
			i.e., $\bar\varphi \coloneqq F'(\bar u) \in H_0^1(\Omega)$
			satisfies $\bar u = \sgn(\bar\varphi)$.
			We further assume that $\bar\varphi$ can be extended to a function $\bar\varphi \in C^1(\R^d)$
			such that \cref{asm:standing} is satisfied,
			i.e., $\nabla\bar\varphi \ne 0$ on $\set{\bar\varphi = 0} \cap \bar\Omega$.
		\item
			\label{asm:everything:4}
			We assume that the parameters $\alpha$, $p$
			are chosen
			such that the assumptions of \cref{thm:main_inequality} are satisfied
			and such that \cref{lem:compact}
			yields compactness of $S'(u) \colon L_0^{\alpha,p}(\Omega)\dualspace \to L^2(\Omega)$.
	\end{enumerate}
\end{assumption}
The required extension of $\bar\varphi$
follows from Whitney's extension theorem
if we have $\bar\varphi \in C^1(\bar\Omega)$,
see \cite[Theorem~2.7]{WachsmuthWachsmuth2025:1}
for an introduction.

We mention that we can always choose
\begin{equation*}
	\alpha = \frac12 + \varepsilon
	\quad\text{and}\quad
	p = 2
\end{equation*}
for $\varepsilon > 0$ small enough
or
\begin{equation*}
	\alpha = 1 \quad\text{and}\quad p = \frac{2 d}{d + 1} < 2
	.
\end{equation*}

\begin{lemma}
	\label{lem:verify_abstract}
	If \cref{asm:everything} holds,
	\cref{asm:abstract}
	is satisfied by
	the choice $X = L_0^{\alpha,p}(\Omega)\dualspace$.
	Further, the non-degeneracy condition \eqref{eq:NDC}
	holds.
\end{lemma}
\begin{proof}
	\hyperref[asm:abstract:1]{\namecrefs{asm:abstract}~\ref*{asm:abstract}\ref*{asm:abstract:1}}
	and 
	\ref{asm:abstract:2} are clear.
	Since $X$ is a reflexive space,
	\itemref{asm:abstract:4} follows from \cref{lem:compact}
	and \eqref{eq:second_derivative_F}.
	The non-degeneracy condition follows from the combination of
	\cref{thm:main_inequality}
	with \cref{lem:verification_NDC}.

	It remains to check
	\itemref{asm:abstract:3}.
	To this end, let sequences
	$\seq{t_k} \subset (0,\infty)$,
	$\seq{h_k} \subset X$ be given
	such that
	$t_k \to 0$, $h_k \weaklystar h \in X$
	and $\bar x + t_k h_k \in \dom G = \Uad$.
	Note that this implies that
	$\norm{t_k h_k}_{L_0^{\alpha,p}(\Omega)\dualspace} \le C t_k$
	and
	$\norm{t_k h_k}_{L^\infty(\Omega)} \le C$.
	We check that $\bar y \coloneqq S(\bar u)$
	and $y_k \coloneqq S(\bar u + t_k h_k)$
	satisfy
	$\norm{y_k - \bar y}_{L^\infty(\Omega)} \to 0$.
	Due to
	\begin{equation*}
		y_k - \bar y
		=
		\int_0^1 S'( \bar u + \xi t_k h_k) (t_k h_k) \d\xi,
	\end{equation*}
	it is sufficient to check that
	we can choose $\lambda \in (0,1)$
	such that
	\begin{equation*}
		\norm{
			S'(u) h
		}_{L^\infty(\Omega)}
		\le
		C
		\norm{h}_{L_0^{\alpha,p}(\Omega)\dualspace}^\lambda
		\norm{h}_{L^\infty(\Omega)}^{1-\lambda}
	\end{equation*}
	holds for all $u,h \in L^\infty(\Omega)$.
	For this, we can argue as in the proof of \cref{thm:stampacchia_meets_bessel},
	but instead of \eqref{eq:stampacchia_2},
	we use
	\begin{equation*}
		\int_\Omega h z_k \d\lambda^d
		\le
		\norm{h}_{L_0^{\alpha,p}(\Omega)\dualspace}^\lambda
		\norm{z_k}_{L_0^{\alpha,p}(\Omega)}^\lambda
		\norm{h}_{L^\infty(\Omega)}^{1-\lambda}
		\norm{z_k}_{L^1(\Omega)}^{1-\lambda}
		.
	\end{equation*}
	If we choose $\lambda > 0$ small enough,
	this leads to the desired inequality.
	Next, we introduce the adjoints
	$\varphi_k \coloneqq F'(\bar u + t_k h_k)$,
	which are the solutions to the adjoint equation \eqref{eq:adjoint}
	with $y = y_k$.
	By considering the adjoint equation,
	it is easy to check that
	$\norm{y_k - \bar y}_{L^\infty(\Omega)} \to 0$
	implies
	$\norm{\varphi_k - \bar \varphi}_{L^\infty(\Omega)} \to 0$.
	By arguing as in the proof of \cite[Lemma~2.6]{Casas2012:1},
	we obtain the inequality
	\begin{equation*}
		\norm{ S'(u) h - S'(\bar u) h}_{L^2(\Omega)}
		\le
		C \norm{S(u) - \bar y}_{L^\infty(\Omega)} \norm{S'(\bar u) h}_{L^2(\Omega)}
		\qquad
		\forall u \in \Uad, h \in L^\infty(\Omega).
	\end{equation*}
	Now, we can repeat the arguments in the proof of \cite[Lemma~2.7]{Casas2012:1}
	to
	obtain
	\begin{equation}
		\label{eq:estimate_second_deriv}
		\abs*{
			F''(u) h^2 - F''(\bar u) h^2
		}
		\le
		C
		\parens*{
			\norm{ S(u) - \bar y }_{L^\infty(\Omega)}
			+
			\norm{ \varphi_u - \bar \varphi }_{L^\infty(\Omega)}
		}
		\norm{S'(\bar u) h}_{L^2(\Omega)}^2
	\end{equation}
	for all $u \in \Uad$ and $h \in L^\infty(\Omega)$,
	where $\varphi_u$ is the adjoint state associated with $u$.

	Since $F$ is twice continuously differentiable on $L^\infty(\Omega)$,
	we get the existence of $\xi_k \in (0,1)$
	such that
	\begin{align*}
		F(\bar u + t_k h_k)
		&=
		F(\bar u)
		+
		t_k \dual{F'(\bar u)}{h_k}
		+
		\frac12 t_k^2 F''(\bar u + \xi_k t_k h_k) h_k^2
		\\&
		=
		F(\bar u)
		+
		t_k \dual{F'(\bar u)}{h_k}
		+
		\frac12 t_k^2 F''(\bar u ) h_k^2
		+
		\frac12 t_k^2 \bracks*{ F''(\bar u + \xi_k t_k h_k) - F''(\bar u)} h_k^2
		.
	\end{align*}
	By using the same arguments leading to $\norm{y_k - \bar y}_{L^\infty(\Omega)} \to 0$,
	we can check $\norm{S(\bar u + \xi_k t_k h_k) - \bar y}_{L^\infty(\Omega)} \to 0$.
	Further, $\norm{S'(\bar u) h_k}_{L^2(\Omega)}$
	is bounded due to \cref{thm:stampacchia_meets_bessel}.
	Together with \eqref{eq:estimate_second_deriv}
	this leads to
	\begin{equation*}
		\abs*{
			F''(\bar u + \xi_k t_k h_k) h_k^2 - F''(\bar u) h_k^2
		}
		\to 0.
	\end{equation*}
	By inserting this in the expansion of $F(\bar u + t_k h_k)$,
	this
	finishes the verification of \eqref{eq:taylor_F}.
\end{proof}

Finally,
we combine the abstract second-order result \cref{thm:second-order_condition}
with the precise structure of the weak-$\star$ second subderivative
from \cref{thm:equal_subderivatives}
to arrive at the main result of this section.
\begin{theorem}
	\label{thm:second_order_for_bang_bang}
	Let \cref{asm:everything} be satisfied.
	Then, the existence of $c,\varepsilon > 0$
	satisfying the
	quadratic growth
	\begin{equation}
		\label{eq:quad_growth_bang_bang}
		F(u)
		\ge
		F(\bar u)
		+
		\frac{c}{2} \norm{u - \bar u}_{L_0^{\alpha,p}(\Omega)\dualspace}^2
		\qquad
		\forall u \in \Uad, \norm{u - \bar u}_{L_0^{\alpha,p}(\Omega)\dualspace} \le \varepsilon
	\end{equation}
	is equivalent to the second-order condition
	\begin{equation}
		\label{eq:ssc_bang_bang}
		F''(\bar u) h^2
		+
		\int_{\set{\bar\varphi = 0} \cap \Omega} \frac{\abs{\nabla\bar\varphi}}{2} h^2 \d\HH^{d-1}
		>
		0
		\qquad
		\forall h \in L^2(\HH^{d-1}|_{\set{\bar\varphi = 0} \cap \Omega} ) \setminus \set{0}
		.
	\end{equation}
\end{theorem}
It is interesting to note that the second-order condition \eqref{eq:ssc_bang_bang}
is independent of the chosen parameters $(\alpha,p)$.
Consequently,
if the quadratic growth \eqref{eq:quad_growth_bang_bang}
holds for one pair of parameters $(\alpha, p)$ satisfying
\itemref{asm:everything:4},
it is satisfied for all parameters
$(\alpha, p)$ satisfying
\itemref{asm:everything:4}.

Similarly,
the second-order condition \eqref{eq:ssc_bang_bang}
coincides with the second-order condition
proved in the setting $X = \MM(\Omega)$ for dimensions $d \le 3$
in
\cite{ChristofWachsmuth2017:1,WachsmuthWachsmuth2021}.
Consequently,
the quadratic growth \eqref{eq:quad_growth_bang_bang}
is equivalent to the quadratic growth in $L^1(\Omega)$
in case $d \le 3$.

Finally, we mention that \eqref{eq:ssc_bang_bang}
with ``$>$'' replaced by ``$\ge$'' is a necessary optimality condition.

\printbibliography

\end{document}